\documentclass[11pt]{amsart}
\usepackage{amsfonts,amssymb,amsmath}\usepackage{verbatim}
\input{xypic}
\newtheorem{theorem}{Theorem}[section]
\newtheorem{lemma}[theorem]{Lemma}
\newtheorem{corollary}[theorem]{Corollary}
\newtheorem{proposition}[theorem]{Proposition}

\theoremstyle{definition}
\newtheorem{definition}[theorem]{Definition}

\theoremstyle{remark}

\theoremstyle{theorem}

\numberwithin{equation}{section}

\renewcommand\leq{\leqslant}\renewcommand\geq{\geqslant}
\newcommand\Z{\ensuremath{\mathbb Z}}
\newcommand\Q{\ensuremath{\mathbb Q}}
\newcommand\C{\ensuremath{\mathbb C}}
\newcommand\Qb{{\overline\Q}}

\newcommand\Aut{\operatorname{Aut}}
\newcommand\Br{\operatorname{Br}}

\newcommand\End{\operatorname{End}}
\newcommand\Ext{\operatorname{Ext}}

\newcommand\Gal{\operatorname{Gal}}
\newcommand\GL{\operatorname{GL}}
\newcommand\Hom{\operatorname{Hom}}

\newcommand\Inf{\operatorname{Inf}}

\newcommand\M{\operatorname{M}}

\newcommand\ord{\operatorname{ord}}

\newcommand\Res{\operatorname{Res}}

\def\O{{\mathcal O}}
\newcommand\acc[2]{\ensuremath{{}^{#1}\hskip-0.1ex{#2}}}
\newcommand\accl[2]{\ensuremath{{{\phantom{\big|}}}^{#1}\hskip-0.2ex{#2}}}

\def\M{\operatorname{M}}
\def\A{\mathcal A}\def\B{\mathcal Q}\def\D{\mathcal D}

\def\p{\mathfrak p}
\newcommand{\comp}{\begin{picture}(6,5)(-3,-2)\put(0,1){\circle{2}} \end{picture}}\def\circ{\comp}

\subjclass[2010]{Primary 11G10, Secondary 11G18, 11F11}

\title{Modular abelian varieties over number fields}

\author{Xavier Guitart}\author{ Jordi Quer}
\address{X. G.: Max Planck Institute for Mathematics, Vivatsgasse 7, 53111 Bonn, Germany and Departament de
Matem\`{a}tica Aplicada II, Universitat Polit\`{e}cnica de Catalunya, C.
Jordi Girona 1-3, 08034 Barcelona, Spain}
\email{xevi.guitart@gmail.com}

\address{J. Q.: Departament de Matem\`{a}tica Aplicada II, Universitat Polit\`{e}cnica de Catalunya, C.
Jordi Girona 1-3, 08034 Barcelona, Spain}
\email{jordi.quer@upc.edu}

\thanks{This research was partially supported by grants MTM2009-13060-C02-01 and  2009SGR-1220}

\begin{document}
\maketitle
\begin{abstract}
The main result of this paper is a characterization of the abelian
varieties $B/K$ defined over Galois number fields with the
property that the $L$-function $L(B/K;s)$ is a product of $L$-functions of non-CM newforms over $\Q$ for congruence
subgroups of the form $\Gamma_1(N)$. The characterization involves the
structure of $\End(B)$, isogenies between the Galois conjugates of
$B$, and a Galois cohomology class attached to $B/K$.

We call the varieties having this property \emph{strongly modular}.
The last section is devoted to the study of a family of abelian surfaces with quaternionic
multiplication.
As an illustration of the ways in which the general results of the paper can be applied
we prove the strong modularity of some particular abelian surfaces belonging to that family, and
we show how to find nontrivial examples of strongly modular varieties by twisting.
\end{abstract}

\section{Introduction}
We will work in the category of abelian varieties up to isogeny,
in which the objects are abelian varieties and the morphisms
between two varieties $A$ and $B$ are the elements of the
$\Q$-vector space $\Hom^0(A,B):=\Q\otimes_\Z\Hom(A,B)$, where
$\Hom(A,B)$ denotes the usual $\Z$-module of homomorphisms between
abelian varieties. In particular,  isogenies become
isomorphisms in our category. We will use the
standard $\Hom,\End$ and $\Aut$ to denote morphisms up to isogeny
(we will suppress the superscripts to lighten the notation). As usual, a field as an index means morphisms
defined over that field. The notation $A\sim B$ will indicate that the abelian varieties
$A$ and $B$ are isogenous, and $A\sim_K B$ that they are isogenous
with an isogeny defined over the field $K$.

Let $f=\sum a_nq^n$ be a weight-two newform for the congruence
subgroup $\Gamma_1(N)$, and let $E_f=\Q(\{a_n\})$ be the number
field generated by its Fourier coefficients. Shimura attaches to
$f$ an abelian variety $A_f$ defined over $\Q$, constructed
as a subvariety of the Jacobian $J_1(N)$ of the modular curve
$X_1(N)$. The variety $A_f$ has dimension equal to the degree
$[E_f:\Q]$, the algebra $\End_\Q(A_f)$ of its endomorphisms
defined over $\Q$ is isomorphic to the number field $E_f$, and the
$L$-function $L(A_f/\Q;s)$ is \emph{equivalent} (i.e., coincides up to a
finite number of Euler factors) with the product $\prod
L(\acc\sigma f;s)$ of the $L$-functions of the Galois conjugates
of the form $f$  \cite[Section 7.5]{shimura-book}.

The abelian varieties $A_f$ and, more generally, all abelian
varieties $A/\Q$ that are isogenous over $\Q$ to  $A_f$ for some $f$ are
 known as \emph{modular abelian varieties}. This modularity
property has many important consequences and applications, for
example:
\begin{itemize}\itemsep=0pt
\item modularity implies the Hasse conjecture for the $L$-function $L(A/\Q;s)$;
\item the theory of Heegner points and results by Gross--Zagier and Kolyvagin produce
    partial results for the variety $A/\Q$
    in the direction of the Birch and Swinnerton-Dyer conjecture;
\item the modularity of Frey's elliptic curves can be used to solve
    certain Diophantine equations of Fermat type.
\end{itemize}
As a result, modular abelian varieties have
been intensively studied and exploited in the last decades. In
practice, one can easily compute and work with modular forms and
the corresponding modular abelian varieties thanks to the powerful
tool provided by the theory of modular symbols: see \cite{cremona}
for elliptic curves and \cite{stein} for arbitrary dimension. The
computer systems {\tt Magma} and {\tt Sage} include packages
 that are able to perform many explicit
computations with those objects.

In the other direction, one would like to characterize the modularity of a given variety $A/\Q$.
In the one-dimensional case, the Shimura--Taniyama conjecture
predicted that every elliptic curve over $\Q$ is modular,
and its proof was completed in \cite{br-co-di-ta} by Breuil, Conrad, Diamond and Taylor, using generalizations and
variants of the ideas and techniques of Wiles \cite{wiles} and Taylor--Wiles  \cite{taylor-wiles}.
In \cite{ribet-avQ}  Ribet introduced the concept of a variety of $\GL_2$-type
as an abelian variety $A/\Q$ for which $\End_\Q(A)$ is a number field of degree equal to the dimension of $A$.
He generalized Shimura--Taniyama by conjecturing that every variety of $\GL_2$-type
is modular over $\Q$, and proved that this fact would be a consequence of Serre's conjecture
on the modularity of $2$-dimensional mod $p$ Galois representations \cite[Theorem 4.4]{ribet-avQ}.
After the  proof of Serre's conjecture by Khare and Wintenberger \cite[Theorem 9.1]{kwI}, we now know that modularity
of an abelian variety over $\Q$ is equivalent to the property of being of $\GL_2$-type.

The abelian varieties of $\GL_2$-type are not absolutely simple in
general: they factor up to isogeny as products of varieties
defined over number fields. After some work done by  Elkies in the
one-dimensional case and by  Ribet in general, in \cite{pyle}  Pyle
gave a characterization of the abelian varieties defined over
number fields that appear in the absolute decomposition of abelian
varieties of $\GL_2$-type, which depends on the structure of their
endomorphism algebras and on the existence of isogenies between their Galois
conjugates. She uses the name \emph{building blocks} for them
(also known as elliptic $\Q$-curves in the one-dimensional case)
and  generalizes the use of the term \emph{modular abelian
variety} to refer to an abelian variety $B/K$ defined over a
number field $K$ that is isogenous to a factor of some $A_f$. In
this sense, Ribet's generalization of the Shimura--Taniyama conjecture
predicts that building blocks and absolutely simple modular
abelian varieties are the same things up to isogeny, and now,
after the work of Khare and Wintenberger, this is known to be a fact.

In this way one gets a new family of abelian varieties defined
over number fields and is tempted to use their modularity in the
same way as was done over $\Q$. But the key property of
modularity that is used for many of the applications is the relation of the
$L$-function of the variety with modular forms, and this property
is  invariant only by isogeny defined over the base field. Hence, to use modularity in the context of abelian varieties
over number fields, it is natural to single out a class of abelian varieties that we will term `strongly modular.'

\begin{definition}
Let $K$ be a number field. An abelian variety $B/K$  is
\emph{strongly modular over $K$} if its $L$-function $L(B/K;s)$
is equivalent to a product of $L$-series of newforms over $\Q$ for  $\Gamma_1(N)$.
\end{definition}

In the next section we give a characterization
(Proposition \ref{proposition-characterization-strong-modularity-GL2-over-number-fields})
of strong modularity as the property that the restriction of scalars to $\Q$
is a product of $\GL_2$-type varieties.
This can be taken as an alternative definition of strong modularity.

It is important to notice that strong modularity is a much more restrictive concept
than that of being just a building block.
For example, consider an elliptic curve over the field of rational numbers $B/\Q$.
Then, the extension of scalars of $B$ to any number field $K$ is a building block,
but $B/K$ is strongly modular only in the case that $K/\Q$ is an abelian extension:
in this case we have that $L(B/K;s)=\prod_{\chi}L(f\otimes \chi;s)$,
where $f$ is the newform such that $L(B/\Q;s)=L(f;s)$
and $\chi$ runs over the group of complex characters of $\Gal(K/\Q)$.
Another less trivial example is the following: consider the elliptic curve $B$
defined over the number field $K=\Q(\sqrt{-3})$ by an equation of the type:
$$Y^2=X^3+4aX^2+2(a^2+\sqrt{-3}b)X,\qquad a,b\in\Q.$$
Then $B/K$ is a building block but it is not strongly modular;
in fact, no curve isogenous to $B$ over $\Qb$ and defined over the field $K$ can be strongly modular.
If we enlarge the base field $K$ to $M=\Q(\sqrt{-3},\sqrt{-2})$
then the curve $B/M$ obtained by extension of scalars is still not strongly modular
but its quadratic twist corresponding to the extension $M(\sqrt{2+\sqrt6})$ is strongly modular.
This example is obtained in \cite[\S 3]{luis-jorge} using the results of \cite{quer-LMS} and \cite{quer-JA},
and is applied to the study of a family of Diophantine equations of Fermat type.
The strong modularity of the twisted model is fundamental for this application.
Examples of this same type but in higher dimension are given in the last section of this paper
as application of our results.

Strongly modular abelian varieties over number fields can be studied as easily as modular abelian varieties over $\Q$.
For instance, their $L$-series satisfy the Hasse conjecture,  Heegner cycles 
and Gross--Zagier type results can be used  to compute rational points on them,
and  they can be used for solving Diophantine equations
as in the example cited in the previous paragraph.

Of course, for abelian varieties over number fields there are other concepts of modularity
that associate the $L$-function of the variety with more general modular and automorphic forms.
The case of abelian varieties over totally real number fields
and their relation to Hilbert modular forms is perhaps the best understood situation,
and many results generalizing the classical ones are known.

The purpose of the present paper is to understand and characterize
the abelian varieties over number fields
whose $L$-functions can be obtained in terms of classical elliptic modular forms over $\Q$
for congruence subgroups $\Gamma_1(N)$,
in the precise sense of our definition of strong modularity. This characterization is given in our main theorem, Theorem
\ref{main-theorem}, in terms of the Galois group of the number field and of certain Galois cohomology class attached to
the variety. The case of non-CM
one-dimensional building blocks, i.e., of $\Q$-curves without
complex multiplication, was already studied in \cite{quer-LMS}.
This paper is a generalization of \cite{quer-LMS} to arbitrary
dimension; many ideas and tools  we will use here were already
introduced in \cite{ribet-avQ}, \cite{pyle} and \cite{quer-LMS}.

It would be desirable to have also methods to produce modular
abelian varieties without the use of modular forms, with the
modularity being just a consequence of their arithmetic and geometric
properties. The one-dimensional case is of course the most well
known, and for this reason it has been the main source of
applications up to now: all elliptic curves over $\Q$ are strongly modular.
Over number fields all CM elliptic curves are modular, and  Elkies
proved in \cite{elkies} that non-CM modular elliptic curves are
parameterized up to isogeny by the (non-cusp, non-CM) rational
points of the modular curves $X^*(N)$ quotient of $X_0(N)$ by the
group of Atkin-Lehner involutions, for squarefree values of $N$. A
method to explicitly work-out this parametrization when the moduli
variety $X^*(N)$ is of genus zero or one (this happens for $81$
values of $N$ in the range 2-238) is given in \cite{go-la-par},
and standard conjectures suggest that only a finite number of
isogeny classes is missing out of these genus $\leq1$ moduli
curves. Once one has a modular elliptic curve the results in
\cite{quer-LMS} can be used to characterize the curves in its
isogeny class that are strongly modular, and explicit equations
for them may be obtained using the methods developed in
\cite{quer-JA}.

The non-modular construction of higher dimensional strongly modular abelian varieties
has been performed up to now only over $\Q$. They appear either
 as Jacobians of curves $C/\Q$ for which one is able to
write down enough endomorphisms defined over $\Q$, or  as the
varieties corresponding to certain points in moduli varieties,
especially Shimura curves. In the last section we consider a
family of modular abelian surfaces over number fields that are
strictly non-rational examples, in the sense that they cannot be
obtained by just extending scalars from varieties defined over
$\Q$, and we see how the main theorem of this paper can be used to
distinguish (and to produce) strongly modular examples.

\subsubsection*{Remark on complex multiplication.}
 Shimura proved that a variety $A_f$ has a factor with complex
multiplication if and only if it is isogenous to a power of an
elliptic curve with complex multiplication. This is also
equivalent to the fact that the newform $f$ admits a twist by a
quadratic character whose kernel is  the field of complex
multiplication of the corresponding elliptic curve. The CM case
 requires a special treatment and, except for section 2, in
which the results hold in complete generality, for the rest of the
paper we will tacitly assume that all abelian varieties
considered have no CM-factors up to isogeny; when necessary we
will stress this condition by saying ``\emph{non-CM abelian variety.}''

\subsubsection*{Acknowledgments} Guitart wants to thank the Max Planck Institute for 
Mathematics for their hospitality and financial support during his stay at the Institute, where part of the
present work has been carried out.

\section{Strong modularity and $\GL_2$-type}

The purpose of this section is to show  that an abelian variety $B/K$ is strongly modular if and only if
 the abelian variety $A/\Q$ obtained by restriction of scalars
$A=\Res_{K/\Q}(B)$ is isogenous over $\Q$ to a product of abelian varieties of $\GL_2$-type.
Due to the fact that this last property is the one that plays a key role in this paper,
many statements become simpler if we enlarge the definition of
$\GL_2$-type to include varieties that are not simple over $\Q$.

\begin{definition}
An abelian variety $A/\Q$ is of \emph{$\GL_2$-type} if $\End_\Q(A)$
contains a commutative semisimple sub-$\Q$-algebra of $\Q$-dimension equal to $\dim A$.
\end{definition}

Note that the standard use of ``$\GL_2$-type'' in the literature corresponds to the varieties
that satisfy our definition and are simple.
The relation between the two concepts is analogous to the relation between arbitrary
CM-abelian varieties and the simple ones (cf. \cite[p.~29]{milne-cm}).

\begin{lemma}\label{lemma-GL2-type-is-invariant-by-products}
An abelian variety is of $\GL_2$-type if and only if
all its $\Q$-simple factors are of $\GL_2$-type.
\end{lemma}

\begin{proof}
For an abelian variety $A/\Q$ let $A\sim_\Q A_1^{r_1}\times\cdots\times A_n^{r_n}$
be its decomposition up to $\Q$-isogeny into $\Q$-simple factors.
Put $\D_i=\End_\Q(A_i)$, let $F_i$ be the center of $\D_i$ and let $t_i=[\D_i:F_i]^{1/2}$
be its index.
The decomposition of $\End_\Q(A)$ into simple algebras is
\begin{equation}\label{decomposition-of-end-A}
\End_\Q(A)\simeq\M_{r_1}(\D_1)\times \cdots\times\M_{r_n}(\D_n),
\end{equation}
and the reduced degree of $\End_\Q(A)$ over $\Q$ is
$[\End_\Q(A):\Q]_{\operatorname{red}}=\sum r_it_i[F_i:\Q]$.

If every $A_i$ is of $\GL_2$-type then $\D_i=F_i$ has degree $[F_i:\Q]=\dim A_i$.
Every field extension $E_i/F_i$ of degree $r_i$ can be embedded in the matrix ring $\M_{r_i}(F_i)$
and the product $\prod E_i$ is a commutative semisimple subalgebra of $\End_\Q(A)$ of dimension
$\sum [E_i:\Q]=\sum r_i[F_i:\Q]=\sum r_i\dim A_i=\dim A$,
hence $A$ is of $\GL_2$-type.

For the converse we will make use of the following basic facts about associative
algebras: for any semisimple $k$-algebra \def\A{\mathcal A}$\A$
the maximal commutative semisimple subalgebras $E\subseteq\mathcal A$ have dimension
$\dim_k E=[\A:k]_{\operatorname{red}}$,
and for every faithful $\A$-module $M$ one has
$\dim_k M\geq[\A:k]_{\operatorname{red}}$ with the equality being
possible only if all the simple subalgebras of $\A$ are matrix algebras
over fields (cf. \cite[Propositions 1.3 and 1.2]{milne-cm}).
The second fact applied to the space of tangent vectors $\operatorname{Lie}(B/\Q)$
of an abelian variety $B/\Q$  gives the inequality
$[\End_\Q(B):\Q]_{\operatorname{red}}\leq\dim B=\dim \operatorname{Lie}(B/\Q)$.

Assume now that $A$ is of $\GL_2$-type.
Let $E\subseteq\End_\Q(A)$ be a commutative semisimple subalgebra with $[E:\Q]=\dim A$.
Then by the previous results
$$\dim A=[E:\Q]\leq[\End_\Q(A):\Q]_{\operatorname{red}}\leq\dim A.$$
Hence each step is an equality and 
$\End_\Q(A)$ is product of matrix algebras over fields; i.e.,  $t_i=1$ for all $i$.

Now, using the inequalities $[\End_\Q(A_i):\Q]_{\operatorname{red}}\leq\dim A_i$ for every $i$,
we have
$$\dim A=[\End_\Q(A):\Q]_{\operatorname{red}}=\sum r_i[F_i:\Q]
    =\sum r_i[\End_\Q(A_i):\Q]_{\operatorname{red}}\leq\sum r_i\dim A_i=\dim A;$$
 the equality at each summand follows,
from which one deduces $[F_i:\Q]=\dim A_i$ for all $i$
and so all simple factors $A_i$ are of $\GL_2$-type.
\end{proof}

\begin{proposition}\label{proposition-characterization-strong-modularity-GL2-over-Q}
An abelian variety $A/\Q$ is strongly modular over $\Q$
if and only if it is of $\GL_2$-type.
\end{proposition}

\begin{proof}
If $A/\Q$ is of $\GL_2$-type, by the previous lemma we have that
$A\sim_\Q A_1^{r_1}\times\cdots\times A_n^{r_n}$, where the
$A_i$'s are $\Q$-simple abelian varieties of $\GL_2$-type.  Results of Ribet \cite[Theorem 4.4]{ribet-avQ}
and Khare and Wintenberger \cite[Theorem 10.1]{kwI}, together with Faltings's isogeny theorem imply the existence of
newforms $f_i$ such that
$A_i\sim_\Q A_{f_i}$. Then $L(A/\Q,s)\sim\prod
L(A_i/\Q,s)^{r_i}\sim\prod L(A_{f_i}/\Q,s)^{r_i}$, where $\sim$ denotes equivalence of $L$-functions. Since each
$L(A_{f_i}/\Q,s)$ is the product of the $L$-functions of the
newforms that are Galois conjugates of $f_i$, the variety $A$ is
strongly modular over $\Q$.

Now we prove the converse.
Let $A/\Q$ be a strongly modular abelian variety over $\Q$,
and let $f_1,\dots,f_n$ be newforms such that $L(A/\Q,s)=\prod L(f_i,s)$.
Let $E_i$ be the field of Fourier coefficients of $f_i$,
and denote by $E=E_1 E_2\cdots E_n$ the composition.
Let $m=[E:\Q]$ and $m_i=[E:E_i]$, and denote by $\Sigma_E$ and $\Sigma_{E_i}$
the corresponding sets of  complex embeddings.
For every index $i$, the restriction of all the elements of $\Sigma_E$ to the field $E_i$
gives $m_i$ copies of every element of $\Sigma_{E_i}$.

We will make use of the following notation:
if $S=\sum a_n n^{-s}$ is a Dirichlet series with $a_n\in \C$ and $\sigma\in\Aut(\C)$,
we denote by $\acc\sigma S$ the series $\sum\acc\sigma a_n n^{-s}$;
that is, the series obtained by applying $\sigma$ to the coefficients $a_n$.
Note that since $L(A/\Q,s)$ has rational coefficients we have that $\acc\sigma L(A/\Q,s)=L(A/\Q,s).$
One has
\begin{multline*}
L(A^m/\Q,s)=L(A/\Q,s)^m=\prod_{\sigma\in\Sigma_E}\acc\sigma L(A/\Q,s)
    =\prod_{\sigma\in\Sigma_E}\prod_{i=1}^n\acc\sigma L(f_i,s)\\
=\prod_{i=1}^n\prod_{\sigma\in\Sigma_E}L(\acc\sigma f_i,s)
    =\prod_{i=1}^n\prod_{\sigma\in\Sigma_{E_i}}L(\acc\sigma f_i,s)^{m_i}
    =\prod_{i=1}^n L(A_{f_i}/\Q,s)^{m_i}=L\left(\bigg(\prod_{i=1}^n A_{f_i}^{m_i}\bigg)/\Q,s\right).
\end{multline*}
Then by Faltings's  isogeny theorem the varieties $A^m$ and $\prod A_{f_i}^{m_i}$
are isogenous over $\Q$. By the uniqueness of decomposition up to $\Q$-isogeny into the product of $\Q$-simple
varieties
it follows that $A$ is isogenous over $\Q$
to a product $\prod A_{f_i}^{e_i}$ for some exponents $e_i\geq0$, and thus it is of $\GL_2$-type.
\end{proof}

For other number fields strong modularity can be reduced to
that of the restriction of scalars.

\begin{proposition}\label{proposition-characterization-strong-modularity-GL2-over-number-fields}
An abelian variety $B/K$ over a number field $K$ is strongly modular over $K$
if and only if $\Res_{K/\Q}(B/K)$ is of $\GL_2$-type.
\end{proposition}

\begin{proof}
The equality of $L$-functions $L(B/K,s)=L((\Res_{K/\Q}B)/\Q,s)$
implies that $B$ is strongly modular over $K$ if and only if
$\Res_{K/\Q}B$ is strongly modular over $\Q$,
and by the previous proposition this is the case if and only if
$\Res_{K/\Q}B$ is of $\GL_2$-type.
\end{proof}

Combining Lemma \ref{lemma-GL2-type-is-invariant-by-products}
with this proposition one immediately obtains the following consequence.

\begin{corollary}
An abelian variety is strongly modular over a number field $K$ if and only if
all its $K$-simple factors are strongly modular over $K$.
\end{corollary}

\section{$\Q$-abelian varieties}

The absolutely simple factors up to isogeny of non-CM abelian varieties
of $\GL_2$-type are studied by  Ribet in \cite{ribet-avQ} and by
 Pyle in \cite{pyle}. A common property of the non-CM abelian varieties of
$\GL_2$-type and of their simple factors is that the object
consisting of the variety together with  its endomorphisms has as
field of moduli the field of rational numbers. In order to deal
with this property the following definitions are useful. For a
given abelian variety $B/\Qb$ and Galois automorphism $\sigma\in
G_\Q$, an isogeny $\mu_\sigma\colon\accl\sigma B\to B$ is said to
be \emph{compatible with the endomorphisms of $B$} if the map
$\End(B)\rightarrow \End(B)\colon\psi\mapsto\mu_\sigma\circ\acc\sigma\psi\circ\mu_\sigma^{-1}$ is
the identity, i.e., if the
 diagram 
$$\diagram\accl\sigma B\rto^{\mu_\sigma}\dto_{\acc\sigma\psi}&B\dto^\psi\\
    \accl\sigma B\rto_{\mu_\sigma}&B\enddiagram$$
is commutative for every $\psi\in\End(B)$.

\begin{definition}[{\cite[p.~194]{pyle}}]\label{definition-Q-variety}
A \emph{$\Q$-abelian variety} (or just \emph{$\Q$-variety} for short) is an abelian variety $B/\Qb$ such that
for every $\sigma\in G_\Q$ there exists an isogeny $\mu_\sigma\colon\accl\sigma B\to B$ compatible with $\End(B)$.
\end{definition}

Let $F$ be the center of the endomorphism algebra $\End(B)$. It is
easily seen that if the isogeny $\mu_\sigma\colon\accl\sigma B\to
B$ is compatible with $\End(B)$, then all  isogenies between these two
varieties compatible with $\End(B)$ are the maps of the form $\psi\circ\mu_\sigma$ for
$\psi\in F^*$; also, if $B$ is a $\Q$-abelian variety, then all
its endomorphisms belonging to the center $F$ are defined over
every field of definition for $B$.

\subsection*{ The cocycle class $[c_B]$.}
Let $B/\Qb$ be a $\Q$-variety. Since the variety is defined over some number field
one can always choose a set of isogenies $\{\mu_\sigma\colon \accl\sigma B \rightarrow  B\}_{\sigma\in G_\Q}$
compatible with the endomorphisms of $B$ that is locally constant. Let $c_B$ be the map defined as
\begin{equation}\label{definition-c_B}
c_B\colon G_\Q\times G_\Q\to F^*,\qquad
    c_B(\sigma,\tau)=\mu_\sigma\circ\acc\sigma\mu_\tau\circ\mu_{\sigma\tau}^{-1}.
\end{equation}
In the following lemma we state some of the properties of $c_B$, which can be straightforwardly checked.

\begin{lemma}
The map $c_B$ is a well-defined continuous $2$-cocycle on $G_\Q$ with values in the group $F^*$,
considered as a $G_\Q$-module with trivial action,
and the cohomology class $[c_B]\in H^2(G_\Q,F^*)$ does not depend on the locally
constant set of isogenies used to define the cocycle $c_B$.
Moreover, the class $[c_B]$  depends only on the $\Qb$-isogeny class of $B$.
\end{lemma}

This invariant $[c_B]$ gives the obstruction to descend up to isogeny the variety and its
endomorphisms over a given number field.
The following result is stated as Proposition 5.2 in \cite{pyle} for $\Q$-abelian varieties
that are building blocks,
but in the proof given there the structure of the endomorphism algebra plays no role,
so we state it here in full generality.

\begin{proposition}[Ribet--Pyle]\label{proposition-cB-restriction}
Let $B/\Qb$ be a $\Q$-variety and let $K$ be a number field.
There exists an abelian variety defined over $K$ and with all the endomorphisms defined over $K$
that is isogenous to $B$ if and only if $[c_B]$ belongs to the kernel of the restriction map
$\Res\colon H^2(G_\Q,F^*)\to H^2(G_K,F^*)$.
\end{proposition}

Generalizing the terminology introduced in \cite{quer-LMS} for
elliptic $\Q$-curves we will say that a $\Q$-variety $B$ is
\emph{completely defined} over a Galois number field $K$ if the
variety $B$ and all its endomorphisms are defined over $K$, and  there exist  isogenies
between Galois conjugates compatible with $\End(B)$ that are all of them defined over $K$. In this
case, let $\mu_s\colon\acc s B\to B$ be an isogeny compatible with the endomorphisms of $B$ for
each element $s\in G=\Gal(K/\Q)$. Then, as for
\eqref{definition-c_B}, one sees that the map
\begin{equation}\label{definition-cocycle-c_B/K-for-completely-defined}
c_{B/K}\colon G\times G\to F^*,\qquad c_{B/K}(s,t)=\mu_s\circ\accl s\mu_t\circ\mu_{st}^{-1}
\end{equation}
is a well-defined $2$-cocycle on $G$ with values in the trivial
$G$-module $F^*$ whose cohomology class $[c_{B/K}]\in
H^2(K/\Q,F^*)$ is an invariant of the $K$-isogeny class of the
variety $B$.

\begin{proposition}\label{proposition-cB-inflation}
Let $B/\Qb$ be a $\Q$-variety and let $K$ be a Galois number field.
There exists an abelian variety completely defined over $K$ that is isogenous to $B$
if and only if $[c_B]$ belongs to the image of the
inflation map $\Inf\colon H^2(K/\Q,F^*)\to H^2(G_\Q,F^*)$.

Moreover, if $[c_B]=\Inf([c])$ for some cocycle class $[c]\in H^2(K/\Q,F^*)$,
then there exists such a variety $B_0/K$ such that $[c_{B_0/K}]=[c]$.
\end{proposition}

\begin{proof}
Since the image of the inflation lies in the kernel of the restriction,
by  Proposition~\ref{proposition-cB-restriction} we can suppose
that $B$ and all of its endomorphisms are defined over $K$.

Assume that $[c_B]=\Inf([c])$. Modifying the $2$-cocycle $c$
by a coboundary we can assume that it is normalized, i.e., takes
the value $c(1,1)=1$, and as a consequence of the cocycle
condition this implies that also $c(s,1)=c(1,s)=1$ for every
$s\in\Gal(K/\Q)$. Moreover, by changing the choice of
isogenies compatible with $\End(B)$ used in \eqref{definition-c_B}
to define $c_B$, we can also suppose that $\inf(c)$ coincides
with the cocycle $c_B$. This implies that $c_B(\sigma,\tau)=1$
whenever  $\sigma$ or $\tau$ belong to the subgroup $ G_K$. It
follows that the map $\sigma\mapsto\mu_\sigma$ is a one-cocycle on
the group $G_K$ with values in the group $\Aut(B)$, viewed as a
module with the natural Galois action of $G_K$, which is in fact
the trivial action since all the elements of $\End(B)$ are defined
over $K$.

Let $B_0$ be the twist of $B$ by this one-cocycle:
it is an abelian variety $B_0$ defined over $K$ together with an isogeny
$\kappa\colon B\to B_0$
such that $\mu_\sigma=\kappa^{-1}\circ{^\sigma \kappa}$ for all $\sigma\in G_K$.
We will see that this variety satisfies the conditions of the proposition.

Every endomorphism of $B_0$ is of the form $\kappa\circ\psi\circ\kappa^{-1}$
for some $\psi\in\End(B)$.
Since all endomorphisms of $B$ are defined over $K$ and the isogenies $\mu_\sigma$ are compatible with $\End(B)$,
for every $\sigma\in G_K$ one has
$$\acc\sigma(\kappa\circ\psi\circ\kappa^{-1})
    =\acc\sigma\kappa\circ\acc\sigma\psi\circ\acc\sigma\kappa^{-1}
    =\kappa\circ\mu_\sigma\circ\acc\sigma\psi\circ\mu_\sigma^{-1}\circ\kappa^{-1}
    =\kappa\circ\psi\circ\kappa^{-1}$$
and the endomorphisms of $B_0$ are defined over $K$ too.

A calculation shows that the maps $\nu_\sigma:=\kappa\circ\mu_\sigma\circ\acc\sigma\kappa^{-1}$
are  isogenies $\accl\sigma B_0\to B_0$ compatible with $\End(B_0)$ for every $\sigma\in G_\Q$,
and the relation of $\mu_\sigma=\kappa^{-1}\circ\acc\sigma\kappa$ for elements $\sigma\in G_K$
shows that $\nu_\sigma=1$ for the $\sigma$ fixing the field $K$.
The cocycle $c_{B_0}$ computed from this set of isogenies is related to $c_B$ by
$c_{B_0}(\sigma,\tau)=\kappa\circ c_B(\sigma,\tau)\circ\kappa^{-1}$ for all $\sigma,\tau\in G_\Q$.
Since $c_B$ is the inflation of $c$ and this cocycle is normalized, one deduces that
$c_{B_0}(\sigma,\tau)=1$ if either $\sigma$ or $\tau$ belong to the subgroup $G_K$.
Applying this fact to a pair $\sigma\in G_\Q$ and $\tau\in G_K$ one deduces that
$$c_{B_0}(\sigma,\tau)=\nu_\sigma\circ\acc\sigma\nu_\tau\circ\nu_{\sigma\tau}^{-1}
    =\nu_\sigma\circ\nu_{\sigma\tau}^{-1}=1\quad\Rightarrow\quad\nu_{\sigma\tau}=\nu_\sigma,$$
which means that $\nu_\sigma$  depends only on the action of $\sigma$ on $K$
(i.e., on the class of $\sigma$ modulo the normal subgroup $G_K$).
Now, applying the identity to a pair $\sigma\in G_K$ and $\tau\in G_\Q$ one has
$$\nu_\sigma\circ\acc\sigma\nu_\tau\circ\nu_{\sigma\tau}^{-1}
    =\acc\sigma\nu_\tau\circ\nu_{\sigma\tau}^{-1}
    =\acc\sigma\nu_\tau\circ\nu_{\tau}^{-1}=1\quad\Rightarrow\quad\acc\sigma\nu_\tau=\nu_\tau$$
proving that the  isogenies $\nu_\sigma$ are also defined over $K$
for every $\sigma\in G_\Q$.

Finally, for every element $s\in\Gal(K/\Q)$ let $\nu_s$  be the
isogeny $\nu_\sigma$ for any $\sigma\in G_\Q$ whose action on $K$
is given by the element $s$. In this way one obtains a set of
 isogenies compatible with  $\End(B_0)$ defined over the field $K$ and the
cocycle $c_{B_0/K}$ computed using this set is the cocycle
$c_{B_0/K}(s,t)=\kappa\circ c(s,t)\circ\kappa^{-1}$. Hence, under
the isomorphism between the centers of the endomorphisms of the varieties $B_0$ and $B$
given by conjugation by the isogeny $\kappa$ between them, the
cohomology class $[c_{B_0/K}]$ is the class $[c]$ we started
with.
\end{proof}

\subsection*{ Simple $\Q$-varieties of the first kind.}
Up to now we have put no restrictions in the $\Q$-abelian
varieties considered; in particular we have not assumed the
varieties to be simple. If $B\sim\prod B_i^{m_i}$ is the
decomposition up to isogeny into simple varieties then it is easy
to see that $B$ is a $\Q$-abelian variety if and only if all its
simple factors have this property, but we will not need this fact
here. For the study of the modular abelian varieties we are
interested in, the case of interest is when the center $F$ of
$\End(B)$ is a totally real number field; this is equivalent to
say that the previous decomposition has a unique simple factor
that is a variety of the first kind in the standard terminology
employed for the classification of simple abelian varieties
according to the type of their endomorphism algebras as algebras
with involution (cf. \cite[p. 193]{mu}). We also recall that the endomorphism algebras of
simple varieties of the first kind are either a totally real
number field (type I varieties) or a quaternion algebra over such
field, that may be either totally indefinite (type II) or totally
definite (type III). So we assume from now on that $F=Z(\End(B))$
is a totally real number field; we could also assume that $B$ is
absolutely simple but in fact that is not necessary and everything
below works for varieties that are powers of simple varieties of
the first kind.

Generalizing to our situation the definitions given
first by Ribet in \cite[p.~113]{ribet-cosdef} for simple varieties of type I,
and then by  Pyle in \cite[p.~218]{pyle} for building blocks,
one can define for every  isogeny $\mu_\sigma\colon\acc\sigma B\to B$ compatible with $\End(B)$
its ``degree'' $\delta(\mu_\sigma)$, which is a totally positive element of the field $F$
whose reduced norm as an element of the $\Q$-algebra $\End(B)$
is the usual degree of the isogeny $\mu_\sigma$.
Moreover, from the definition of this map and the fact that the Rosati involution
fixes the elements of the center $F$ one gets, exactly as in \cite{ribet-cosdef,pyle},
the following identity:
\begin{equation}\label{csquare-equal-degree}
c_B(\sigma,\tau)^2=\delta(\mu_\sigma)\delta(\mu_\tau)\delta(\mu_{\sigma\tau})^{-1},
\end{equation}
showing that the cohomology class $[c_B]$ belongs to the $2$-torsion subgroup
$H^2(G_\Q,F^*)[2]$.

Now, the structure of the group $H^2(G_\Q,F^*)[2]$ is particularly simple and
a number of consequences about fields of definition can be deduced just by looking at it.
As it is des\-cri\-bed in \cite[p.~114]{ribet-cosdef} (see also \cite[Section 2]{quer-MC}),
if one starts with any group isomorphism $F^*\simeq\{\pm1\}\times F^*/\{\pm1\}$,
and using basic facts of group cohomology, one obtains a decomposition
\begin{equation}\label{decomposition-sign-degree}
H^2(G_\Q,F^*)[2]\simeq H^2(G_\Q,\{\pm1\})\times\Hom(G_\Q,F^*/\{\pm1\}F^{*2})
\end{equation}
under which every $2$-torsion cohomology class $\xi\in H^2(G_\Q,F^*)$ 
has two components $\xi=(\xi_\pm,\overline\xi)$.
The \emph{sign} component $\xi_\pm\in H^2(G_\Q,\{\pm1\})\simeq\Br_2(\Q)$ is an element of
the $2$-torsion of the Brauer group of $\Q$.
The \emph{degree} component $\overline\xi$ is a group homomorphism $G_\Q\to F^*/\{\pm1\}F^{*2}$.
Note that the decomposition of the cohomology group depends on the decomposition
of the (trivial) $G_\Q$-module $F^*$ we have chosen,
but it is easy to see that the degree component does not depend on it,
and also that, for the classes $[c_B]$ attached to $\Q$-varieties $B$,
the degree component is just the map $\sigma\mapsto\delta(\mu_\sigma)\mod\{\pm1\}F^{*2}$
(hence the name).

Given an element $\xi\in H^2(G_\Q,F^*)[2]$ we will denote by $K_P$
the field fixed by the kernel of the degree component $\overline\xi$;
since this morphism takes values in a $2$-torsion group,
the field $K_P$ is an abelian extension of exponent $2$ of the field $\Q$.

\begin{proposition}
Let $B$ be a $\Q$-abelian variety with $F=Z(\End(B))$ a totally real number field.
Let $K_P$ be the field fixed by the kernel of $\overline{[c_B]}$.
\begin{enumerate}\itemsep=0pt
\item If $B_0\sim B$ with $B_0$ and $\End(B_0)$ defined over $K$ then $K_P\subseteq K$.
\item There exist isogenous varieties $B_0\sim B$ defined over fields of the form
    $K=K_P\cdot\Q(\sqrt a)$ for some $a\in\Q$, with $\End(B_0)$ also defined over $K$.
\item There exist isogenous varieties $B_0\sim B$ completely defined over fields
    of the form $K=K_P\cdot\Q(\sqrt a,\sqrt b)$ for some $a,b\in\Q$.
\end{enumerate}
\end{proposition}

\begin{proof}
The decomposition \eqref{decomposition-sign-degree} has analogues
for the group $H^2(G_K,F^*)[2]$ for every number field $K$, and for the
group $H^2(K/\Q,F^*)[2]$ for every Galois number field $K$. The restriction and inflation maps respect the corresponding
decompositions. It follows that the class $[c_B]$ belongs to the
kernel (resp. the image) of the restriction to the first group
(resp. of the inflation from the second group) if and only if the
two components sign and degree belong to the corresponding kernels
(resp. images).

As for the degree component
$\overline{[c_B]}\in\Hom(G_\Q,F^*/\{\pm1\}F^{*2})$ each of the two
conditions either on the inflation or on the restriction are
equivalent to the fact that $K_P\subseteq K$. Every element of
$H^2(G_\Q,\{\pm1\})\simeq\Br_2(\Q)$ can be identified with a
quaternion algebra, that can be written as a pair $(a,b)_\Q$ with
$a,b\in\Q^*$, using the standard notation. Such an element always
can be trivialized by restriction to a (at most) quadratic
extension, for example the extension $\Q(\sqrt a)$. Also, this
element can be inflated from a cohomology class defined on the
Galois group of a (at most) biquadratic extension, for example the
extension $\Q(\sqrt a,\sqrt b)$ (see \cite[Section 2]{quer-JA}).
\end{proof}

In \cite{ribet-cosdef}  Ribet proved that for varieties with
$\End(B)=F$ a totally real number field of odd degree $[F:\Q]=\dim
B$ the field $K_P$ already trivializes the sign component, so that
there are always $B_0\sim B$ with $\End(B_0)$ defined over $K_P$.
It is easily seen that his argument works in fact for all
$\Q$-varieties of the first kind of odd dimension without the
assumption on the size of their endomorphisms. On the contrary,
for even dimension this fact is not true any more, as the modular
examples given in \cite{quer-MC} show.

\section{$K$-building blocks}

Since we want to study abelian varieties over a number field $K$
that are quotients up to $K$-isogeny of varieties of $\GL_2$-type,
we slightly adapt  the definition of ``building block'' given by  Pyle in \cite[p.~195]{pyle}, in order to keep track
of their
decomposition  over $K$ and not merely over $\Qb$.

\begin{definition}\label{definition-K-building-block}
Let $K/\Q$ be a Galois extension.
We say that a (non-CM) abelian variety $B/K$ is a \emph{$K$-building block} if
\begin{enumerate}\itemsep=0pt
\item $B$ is a $\Q$-variety admitting  isogenies $\mu_\sigma\colon\accl\sigma B\to B$ compatible with $\End(B)$
    defined over $K$ for every $\sigma\in G_\Q$, and
\item $\End_K(B)$ is a division algebra with center a number field $E$,
    having index $t\leq 2$ and reduced degree $t[E:\Q]=\dim B$.
\end{enumerate}
\end{definition}

\noindent
We note the following remarks:
\begin{itemize}\itemsep=0pt
\item The requirement that $\End_K(B)$ is a division algebra implies that
     $K$-building blocks are $K$-simple abelian varieties, but they may factor over larger fields.
\item The $\Q$-building blocks are the (non-CM) $\Q$-simple abelian varieties of $\GL_2$-type.
\item The $\Qb$-building blocks are the building blocks in the sense of Pyle's definition;
    we will also use this terminology without a prefix field sometimes.
\item For $\Qb$-building blocks the field $E$ is necessarily a totally real field,
    equal to the center of $\End(B)$.
    In general it may be either totally real or a CM-field,
    and the center $F$ of $\End(B)$ is a (necessarily totally real) subfield of $E$.
\item We do not require $B$ to have all its endomorphisms defined over $K$. This means that
    a $K$-building block is not necessarily a $\Q$-variety completely defined over $K$.
\end{itemize}

\begin{proposition}
Let $A/\Q$ be a $\Q$-simple abelian variety of $\GL_2$-type without CM and let $K/\Q$ be a Galois extension.
Then the extension of scalars $A/K$ is $K$-isogenous to a power $B^n$ of a $K$-building block $B/K$.
\end{proposition}

\begin{proof}
Let $A\sim_K B_1^{n_1}\times\cdots \times B_r^{n_r}$ be the
decomposition up to $K$-isogeny into $K$-simple varieties. Since
$\End_\Q(A)$ is a subfield of $\End(A)$ it acts on each isotypical factor $B_i^{n_i}$, hence
$[\End_\Q(A):\Q]\ |\ 2n_i\dim B_i$. But $[\End_\Q(A):\Q]=\dim A
=\sum n_i \dim B_i$, and this implies that
either $[\End_\Q(A):\Q]=n_i\dim B_i$ or $[\End_\Q(A):\Q]=2n_i\dim B_i$ for each index.
The second case is not possible since we are assuming that no
subvariety of $A$ has CM, and so $[\End_\Q(A):\Q]=n_i\dim B_i$,
which implies that there is only one isotypical factor and $A\sim_K B^n$.

Next, we prove that $\End_\Q(A)$ is a maximal subfield of $\End_K(A)$ or, equivalently,
that $\End_\Q(A)$ is its own centralizer in $\End_K(A)$. Let $\varphi$
be an element of $\End_K(A)$ that commutes with $\End_\Q(A)$. The
image $\varphi(A)$ is isogenous to $B^r$ for some $r$. Since
$\varphi$ commutes with $\End_\Q(A)$, the field $\End_\Q(A)$ acts
on $B^r$, and this implies that $[\End_\Q(A):\Q]\, |\, 2 r\dim B$.
This  gives only two options: either $[\End_\Q(A):\Q]=r\dim B$ or
$[\End_\Q(A):\Q]=2r\dim B$. Again, the second is not allowed since
$B^r$ can not have CM. This means that $r=n$ and $\varphi$ is an
isogeny. Hence, $C(\End_\Q(A))$ is a field and therefore
$C(\End_\Q(A))=\End_\Q(A).$

Set $E=Z(\End_K(B))$ and let $t$ be the index of $\End_K(B)$. We
can prove now that $t[E:\Q]=\dim B$. This comes from the
decomposition $A\sim_K B^n$, which translates into an isomorphism
$\End_K(A)\simeq \M_n(\End_K(B))$. Since $\End_\Q(A)$ is a maximal
subfield of $\End_K(A)$ taking dimensions over $E$ we have that
$[\End_\Q(A):E] = n t $, and multiplying both sides of this
expression by $[E:\Q]$ it gives $[\End_\Q(A):\Q]=nt[E:\Q]$. Since
$[\End_\Q(A):\Q]=\dim A=n\dim B$ we see that $\dim B=t[E:\Q]$.

Since $B$ is $K$-simple, $\End_K(B)$ is a division algebra acting
on $H_1(B,\Q)$; therefore $[\End_K(B):\Q]\ | \ 2\dim B$. This
means that $t^2[E:\Q]\ | \ 2 \dim B $, and by the relation
$t[E:\Q]=\dim B$ we see that $t\dim B \ | \ 2 \dim B$, showing
that $t\leq 2$.

Finally, we have to show that $B$ is a $\Q$-variety with isogenies between Galois conjugates
defined over $K$. The argument is the same as in the
proof of  Proposition 1.4 in \cite{pyle}, but starting from an
isogeny between $A_K$ and $B^n$ defined over $K$. It only has to be noticed that the
isogenies $\alpha(\sigma)\colon A\to A$ that appear in that
proof are defined over $\Q$, and this implies that the
isogenies defined on page 195 of \cite{pyle} are defined over $K$.
\end{proof}

\begin{corollary}\label{corollary-necessity-of-K-bb}
If a $K$-simple variety is strongly modular over a Galois number field $K$,
then it is a $K$-building block.
\end{corollary}

\begin{proof}
Let $B$ be a $K$-simple strongly modular variety.
By Proposition \ref{proposition-characterization-strong-modularity-GL2-over-Q},
 since $B$ is strongly modular over $K$, $\Res_{K/\Q}(B)$
is of $\GL_2$-type.
Since $\Res_{K/\Q}(B)\sim_K\prod_{s\in\Gal(K/\Q)}\accl s B$,
the variety $B$ is a $K$-simple factor  of
a $\Q$-simple variety of $\GL_2$-type,
and then it is a $K$-building block by the previous proposition.
\end{proof}

The converse of this corollary is not true: a $K$-building block
needs extra conditions to be strongly modular.
This conditions are related to a cohomology class $[c_{B/K}]$ attached to $B$ that
is defined in a similar way than \eqref{definition-cocycle-c_B/K-for-completely-defined} as follows.
Let $B$ be a $K$-building block over a Galois number field $K$;
put $G=\Gal(K/\Q)$,  $E=Z(\End_K(B))$ and $F=Z(\End(B))$.
Let $\{\mu_\sigma\}_{\sigma\in G_\Q}$ be a  set of isogenies compatible with $\End(B)$ defined over $K$.
For each $s\in G$ choose a representative $\tilde{s}$ in $G_\Q$, and define
$$c_{B/K}\colon G\times G\to E^*,\qquad c_{B/K}(s,t)
    =\mu_{\tilde s}\circ\acc{\tilde s}{\mu_{\tilde t}}\circ\mu_{\widetilde{st}}^{-1}.$$
When all the endomorphisms of $B$ are defined over $K$ this cocycle $c_{B/K}(s,t)$
coincides with \eqref{definition-cocycle-c_B/K-for-completely-defined}.
Since now we are not requiring the field $K$
to be a field of definition of all the endomorphisms of $B$,
we can only guarantee that $c_{B/K}(s,t)$ lies in $E^*$ but not in $F^*$ as
it happens when the variety is completely defined over $K$.
In the next lemma we state the main properties of this cocycle.

\begin{lemma}
The map $c_{B/K}$ is a 2-cocycle on $G$ with values in $E^*$,
considered as a module with trivial action.
The cohomology class $[c_{B/K}]\in H^2(K/\Q,E^*)$ depends
neither on the lift $s\mapsto \tilde{s}$ nor on the choice of the isogenies $\mu_{\tilde{s}}$.
Moreover, the inflation of $[c_{B/K}]$ to $H^2(G_\Q,E^*)$
coincides with the image of $[c_B]$ under the morphism $H^2( G_\Q,F^*)\to H^2( G_\Q,E^*)$
induced by the embedding $F^*\hookrightarrow E^*$.
\end{lemma}

\begin{proof}
Let $\varphi$ be an element of $\End_K(B)$.
Since $\widetilde{st}=\tilde s\,\tilde t\,\tau$ for some $\tau\in G_K$ we have
\begin{eqnarray*}
c_{B/K}(s,t) \circ \varphi
&=&\mu_{\tilde s} \circ \acc{\tilde s}{\mu_{\tilde t}} \circ \mu_{\widetilde{st}}^{-1} \circ \varphi=
   \mu_{\tilde s} \circ \acc{\tilde s}{\mu_{\tilde t}} \circ \acc{\widetilde{st}}\varphi
        \circ \mu_{\widetilde{st}}^{-1}=\\
&=&\mu_{\tilde s} \circ \acc{\tilde s}{\mu_{\tilde t}} \circ \acc{\tilde s\,\tilde t\,\tau}{\varphi}
        \circ \mu_{\widetilde{st}}^{-1}=
   \mu_{\tilde s} \circ \acc{\tilde s}{\mu_{\tilde t}} \circ \acc{\tilde s\,\tilde t}{\varphi}
        \circ \mu_{\widetilde{st}}^{-1}=\\
&=&\mu_{\tilde s} \circ \acc{\tilde s}{\varphi} \circ \acc{\tilde s}{\mu_{\tilde t}}
        \circ \mu_{\widetilde{st}}^{-1}=
    \varphi \circ \mu_{\tilde s} \circ \acc{\tilde s}{\mu_{\tilde t}}
        \circ \mu_{\widetilde{st}}^{-1}=\varphi \circ c_{B/K}(s,t),
\end{eqnarray*}
and this shows that $c_B(s,t)$ lies in $E$.
In the same way we can prove the cocycle condition,
and the independence on the set $\{\mu_\sigma\}_{\sigma\in G_\Q}$
is seen in an analogous way than for the case of the cocycle $c_B$.

Observe that, if $\sigma\in G_\Q$ is such that $\sigma|_{K}=\tilde
s|_{K}$ then $\mu_\sigma\circ\mu_{\tilde{s}}^{-1}$ commutes with
the elements in $\End_K(B)$; therefore, we can write
$\mu_\sigma=\lambda_\sigma\circ\mu_{\tilde{s}}$ for some
$\lambda_\sigma\in E^*$. Using this it is immediate to see that
the use of another lift from $\Gal(K/\Q)$ to $ G_\Q$ would modify
the cocycle $c_{B/K}$ by a coboundary.

It remains to prove the last statement in the lemma. Take
$\sigma,\tau\in G_\Q$ and put $s=\sigma|_{K}$, $t=\tau|_{K}$. We
use the same name for the cocycles and for their images for the
morphisms involved; namely, $c_{B/K}$ is the inflation to $ G_\Q$
of $c_{B/K}$ and $c_B$ is the image of $c_B$ in $Z^2( G_\Q,E^*)$.
By the definitions
$c_B(\sigma,\tau)=\mu_\sigma\circ\acc\sigma{\mu_\tau}\circ\mu_{\sigma\tau}^{-1}$
and $c_{B/K}(\sigma,\tau)=\mu_{\tilde s}\circ\acc{\tilde
s}{\mu_{\tilde t}}\circ\mu_{\widetilde{st}}^{-1}$. Since
$\sigma_{|K}=\tilde{s}_{|K}$ we see that
$\mu_\sigma=\mu_{\tilde{s}}\circ\lambda_\sigma$ for some
$\lambda_\sigma\in E$. Now
$c_B(\sigma,\tau)=c_{B/K}(\sigma,\tau)\circ\lambda_\sigma\circ\lambda_\tau\circ\lambda_{\sigma\tau}^{-1}$
and the two cocycles are cohomologous.
\end{proof}

\subsection*{ Restriction of scalars of $K$-building blocks.}
Our objective now is to compute the endomorphism algebra
$\End_\Q(\Res_{K/\Q}(B))$ for a $K$-building block $B$. What we
obtain is a generalization of the expression found by  Ribet in
\cite[Lemma 6.4]{ribet-avQ} for the case of $\Q$-curves, giving
the algebra as a twisted group algebra. The main difference is
that in our case the algebra is obtained by a construction that
mimics the standard twisted group algebra definition, which we
first describe in abstract terms.

\def\A{\mathcal A}
Let $\A$ be a central $E$-algebra and let $c\in Z^2(G,E^*)$ be a two-cocycle on a finite group
$G$ with values in the multiplicative group $E^*$ viewed as a module with trivial action.
One defines the $E$-algebra $\A^c[G]$ by  generalizing
the usual definition of twisted group algebra:
it is the free left $\A$-module $\oplus_{s\in G}\A\cdot\lambda_s$ with basis a set of symbols
$\lambda_s$ indexed by the elements $s\in G$ and multiplication defined by the relations:
\begin{equation}\label{definition-generalized-twisted-group-algebra}
\begin{aligned}
a\cdot\lambda_s&=\lambda_s\cdot a,\qquad\text{for}\ \ \ a\in\A,\\
\lambda_s\cdot\lambda_t&=c(s,t)\cdot\lambda_{st}.\end{aligned}\end{equation}
The cocycle condition for $c$ is used to check that this definition makes sense
and produces an associative algebra, and of course its isomorphism class does
depend only on the cohomology class of the cocycle $c$.
This algebra is related with the twisted group
algebra $E^c[G]$ through the following isomorphism:
$$\A^c[G]\simeq\A\otimes_E E^c[G]\qquad\text{as}\quad E\text{-algebras}.$$
Indeed, if we let $E^c[G]=\oplus_{s\in G}E\cdot\lambda_s$ then the
map $a\otimes\sum
x_s\cdot\lambda_s\mapsto\sum(ax_s)\cdot\lambda_s\colon\A\otimes_E
E^c[G]\to\A^c[G]$ is an isomorphism of $E$-algebras.

\begin{proposition}\label{proposition-endomorphism-algebra-restriction-of-scalars}
Let $B$ be a $K$-building block over a Galois number $K$ field with
group $G=\Gal(K/\Q)$. Let $\D=\End_K(B)$ and $E=Z(\D)$. Then,
\begin{equation}\label{eq:endomorphism algebra of the restriction of scalars}\End_\Q(\Res_{K/\Q}(B))\simeq\D\otimes_E
E^{c_{B/K}}[G].\end{equation}
\end{proposition}

\begin{proof}
Call $A$ the variety $\Res_{K/\Q}(B)$.
For each $s\in G$ fix a representative $\tilde{s}$ for $s$ in $G_\Q$
 imposing that $\tilde{1}=1$.
Let $\{\mu_\sigma\}_{\sigma\in G_\Q}$ be
a locally constant set of  isogenies compatible with $\End(B)$ defined over $K$
in which we have chosen $\mu_{1}$ to be the identity.
We know that $A\sim_K\prod_{s\in G}\acc{\tilde s}B$, and that by the universal property of the
 restriction of scalars functor $\End_\Q(A)\simeq \Hom_K(A,B)$. Hence,
$$\End_\Q(A)\simeq \Hom_K(A,B)\simeq\prod_{s\in G}\Hom_K(\acc{\tilde s}B,B)\simeq
    \prod_{s\in G}\D\cdot\mu_{\tilde s}$$
and we see that $\End_\Q(A)$ is a left $\D$-module of dimension
$[K:\Q]$. We shall determine now its structure as an algebra.
Define, for $s\in G$, $\lambda_s$ to be the endomorphism of $A$
which sends $\acc{\tilde{ts}}B$ to $\acc{\tilde t}B$ via
$\acc{\tilde t}\mu_{\tilde s}$. It is fixed by all elements in
$G_\Q$ and so it is an endomorphism of $A$ defined over $\Q$.
Since we forced $\tilde{1}$ to be $1$, we can identify $\lambda_1$
with the identity endomorphism of $\End_\Q(A)$.

We can embed $\D$ in $\End_\Q(A)$ by sending each $d\in\D$ to the morphism
whose components are the diagonal maps $\acc{\tilde s}d\colon\acc{\tilde s}B\to\acc{\tilde s}B$.
Hence, we can multiply the $\lambda_s$ by elements $d$ in $\D$ in the following way,
depending on whether we left or right multiply:
$$d\circ\lambda_s\colon\acc{\widetilde{ts}}B
    \buildrel{\acc{\tilde t}{\mu_{\tilde s}}}\over\longrightarrow\acc{\tilde t}B
    \buildrel{\acc{\tilde t} d}\over\longrightarrow\acc{\tilde t}B$$
$$\lambda_s\circ d\colon\acc{\widetilde{ts}}B
    \buildrel{\acc{\widetilde{ts}}d}\over\longrightarrow\acc{\widetilde{ts}}B
   \buildrel{\acc{\tilde t}{\mu_{\tilde s}}}\over\longrightarrow\acc{\tilde t}B.$$
By the compatibility of the isogenies it is clear that these two maps coincide,
and therefore $d\circ\lambda_s=\lambda_s\circ d$.
Also the compatibility of the isogenies gives us the formula
$\lambda_s\circ\lambda_t=c_{B/K}(s,t)\circ\lambda_{st}$.
That is,  multiplication in $\End_\Q(A)$ is given in terms of this basis by formulas
\eqref{definition-generalized-twisted-group-algebra} with cocycle $c_{B/K}$,
so that this algebra is isomorphic to $\D^{c_{B/K}}[G]$.
\end{proof}

\section{Strongly modular abelian varieties}

Let $B$ be a $K$-building block over a Galois number field $K$ with Galois group $G=\Gal(K/\Q)$.
Let $\D=\End_K(B)$, $E=Z(\D)$, and $t$ the index of $\D$.
Recall that in the previous section we have associated to $B/K$
a cohomology class $[c_{B/K}]\in H^2(G,E^*)$.
 In this section we characterize when $B$ is strongly modular over $K$ in terms of that class.

\begin{lemma}\label{lemma-restriction-of-scalars-of-GL2-type}
Let $B$ be a $K$-building block over a Galois number field $K$.
If $A=\Res_{K/\Q}(B)$ is an abelian variety of $\GL_2$-type, then
$$A \sim_\Q A_1^t\times\cdots\times A_n^t,$$
with the $A_i$ pairwise non-isogenous $\Q$-simple abelian varieties of $\GL_2$-type.
\end{lemma}

\begin{proof}
A priori we know that
\begin{equation}\label{eq:1}A\sim_\Q A_1^{r_1}\times\cdots\times A_n^{r_n}\end{equation}
for some $r_i> 0$ and with the $A_i$ being non-isogenous
$\Q$-simple abelian varieties of $\GL_2$-type.
If we set $E_i=\End_\Q(A_i)$ then $\End_\Q(A)\simeq\M_{r_1}(E_1)\times\dots\times\M_{r_n}(E_n)$.

First we show that each $r_i$ is at least $t$. If $t=1$ this is
clear, so we suppose now that $t=2$. There is an injection of
algebras $\End_K(B)\hookrightarrow \End_\Q(A)$, and so $\End_K(B)$
injects into each simple component $M_{r_i}(E_i)$ of $\End_\Q(A)$.
If $t=2$ then $\End_K(B)$ is non-commutative, and so each $r_i$ must be at least 2.

Now we show that each $r_i$ is in fact equal to $t$.
On the one hand we know, by the universal property of the restriction of scalars,
that $\End_\Q(A)\simeq\Hom_K(A_K,B)$,
and using that $A_K\sim_K\prod_{s\in \Gal(K/\Q)}{\acc s B}$ we have that
\begin{equation}\label{eq:2}
\End_\Q A\simeq \Hom_K(A_K,B)\simeq \Hom_K(\prod{^sB},B)\simeq
\bigoplus_{s\in \Gal(K/\Q)}\Hom_K({^sB},B).
\end{equation}
Since $B$ is a $K$-building block each $^sB$ is $K$-isogenous to
$B$, and so we have an isomorphism of $\D$-modules
$\Hom_K({^sB},B)\simeq\D$. Since $\D$ is a $\Q$-vector space of
dimension $t^2[E:\Q]=t\dim B$ one obtains $\dim_\Q \End_\Q(A)=|G|t\dim
B=t\dim A$.

On the other hand, we can use expression \eqref{eq:1} to
calculate the same dimension. We have shown that $r_i\geq t$ for
all $i$. Suppose that for some $i$ we had $r_i>t$. Then we would
find that
\begin{eqnarray*}
\dim_\Q \End_\Q A&=&r_1^2 \dim A_1+\cdots+r_n^2\dim A_n > tr_1\dim
A_1+\cdots+tr_n\dim A_n=\\&=&t(r_1\dim A_1+\cdots +r_n\dim
A_n)=t\dim A,
\end{eqnarray*}
which is a contradiction with the first calculation we made.
\end{proof}

\begin{lemma}\label{lemma-restriction-of-scalars-and-commutativity}
Let $B$ be a $K$-building block over a Galois number field with
$G=\Gal(K/\Q)$. Then $B$ is strongly modular if and only if the
algebra $E^{c_{B/K}}[G]$ is commutative.
\end{lemma}

\begin{proof}
First suppose that $E^{c_{B/K}}[G]$ is commutative.
Then it is a product of fields so that  $E^{c_{B/K}}[G]=\prod E_i$.
Call $A$ the variety $\Res_{K/\Q}(B)$.
By Proposition \ref{proposition-endomorphism-algebra-restriction-of-scalars} we know that
$$\End_\Q(A)\simeq\D\otimes_E E^{c_{B/K}}[G]\simeq\prod\D\otimes_E E_i,$$
with $\D\otimes_E E_i$ a central simple $E_i$-algebra with index
$t_i$ dividing $t$. Corresponding to this decomposition of
$\End_\Q(A)$ there is a decomposition of $A$ up to $\Q$-isogeny:
$A\sim_\Q\prod A_i$, and $\End_\Q(A_i)\simeq\D\otimes_E E_i$. As
$A_K\simeq\prod{^sB}\sim_K B^{|G|}$, each $A_i$ is $K$-isogenous
to $B^{n_i}$ for some $n_i$. We claim that $n_i$ equals $[E_i:E]$.
To prove the claim, first we observe that the natural inclusion
$\End_\Q(A_i)\hookrightarrow \End_K(A_i)$ gives an injective
morphism $\D\otimes_E E_i \hookrightarrow \M_{n_i}(\D).$ Looking
at the reduced degrees of these algebras over $E$ we see that
$t[E_i:E]\leq t n_i$, and then $[E_i:E]\leq n_i$. To see the
equality, we can use that on the one hand, as
$\End_\Q(A)\simeq\bigoplus_{s\in G}\Hom(\accl s
B,B)\simeq\bigoplus_{s\in G}\D$ we have that
$$[\End_\Q(A):E]=|G|t^2=t^2\sum n_i.$$
But, on the other hand we have that
$$[\End_\Q(A):E]=[\D\otimes_E\prod E_i\ :E]=t^2\sum [E_i:E],$$
and this gives that $[E_i:E]=n_i$.

Returning to the proof of the lemma,
since $\End_\Q(A_i)\simeq\D\otimes_E E_i$ is a central simple algebra of index $t_i|t$,
there exists a  division $E_i$-algebra $\D_i$ of index $t_i$ acting on the differentials of $A_i$.
The space of differentials of $A_i$ is a $\Q$-vector space
of dimension equal to the dimension of $A_i$, and so we have that $[\D_i:\Q]|\dim A_i$.
But $[\D_i:\Q]=t_i^2[E_i:E][E:\Q]$ and $\dim A_i=n_i\dim B=n_i t[E:\Q]=t[E_i:E][E:\Q]$,
because $n_i=[E_i:E]$. This means that
$$t_i^2[E_i:E][E:\Q]\ |\ t[E_i:E][E:\Q],$$
so $t_i^2|t$, which implies $t_i=1$. This means that $\D\otimes_E E_i\simeq\M_t(E_i)$,
and therefore $A_i\sim_\Q (A_i')^t$, for some abelian variety $A_i'$  with $\End_\Q(A_i')\simeq E_i$.
Finally, $A_i'\sim_K B^{n_i/t}$, which gives that
$$[E_i:\Q]=n_i[E:\Q]=\frac{n_i}{t}t[E:\Q]=\frac{n_i}{t}\dim B=\dim A_i',$$
showing that each $A_i'$ is a variety of $\GL_2$-type.

In order to prove the other implication, by the previous lemma we
can suppose that $A\sim_\Q A_1^t\times\cdots\times A_n^t$, and as
a consequence that
\begin{equation}\label{eq:3}\End_\Q(A)\simeq\M_t(E_1)\times\cdots\times\M_t(E_n),\end{equation}
where the notation is the same as in the first part of the proof. On the other hand,
$$\End_\Q(A)\simeq\D\otimes_EE^{c_K}[G]=\D\otimes_E\prod\M_{r_i}(C_i)$$
where the $C_i$ are division algebras. But \eqref{eq:3} forces
$r_i=1$ and $C_i\simeq E_i$ for each $i$.
\end{proof}

Now we state our main result, giving a characterization of strong modularity.

\begin{theorem}\label{main-theorem}
Let $K$ be a Galois number field and let $B/K$ be a $K$-simple abelian variety.
Then $B$ is strongly modular over $K$ if and only if
it is a $K$-building block, the extension $K/\Q$ is abelian, and $[c_{B/K}]$ belongs to the subgroup
$\Ext(G,E^*)\subseteq H^2(G,E^*)$ consisting of symmetric cocycle classes.
\end{theorem}

\begin{proof}
By Corollary \ref{corollary-necessity-of-K-bb} being a $K$-building block is a necessary condition,
and in that case the previous lemma says that being strongly modular is equivalent to the fact
that the algebra $E^{c_{B/K}}[G]$ is commutative.
A twisted group algebra $E^c[G]$ is commutative if and only if the group $G$ is abelian
and the cocycle $c$ is symmetric.
\end{proof}

\subsection*{ Strongly modular simple varieties.}
The previous theorem shows that strong modularity puts very restrictive conditions on varieties. In what follows, we
consider the  setting in which $B/\Qb$ is a $\Qb$-building block. Given a Galois number field $K$, using Theorem
\ref{main-theorem} we want to give necessary and sufficient conditions to guarantee the existence, in the
$\Qb$-isogeny class of $B$, of some variety which is completely defined over $K$ and strongly modular over $K$.

For that, let $B$ be a  $\Qb$-building block and let $\D=\End(B)$. The center
$F=Z(\D)$ is a totally real number field and $\D$ is either equal
to $F$, in which case $t=1$ and $[F:\Q]=\dim B$, or it is a
totally indefinite quaternion algebra over $F$, with $t=2$ and
$[F:\Q]=\frac12\dim B$. Let $\xi=[c_B]\in H^2(G_\Q,F^*)$ be the
cohomology class attached to $B$.

We fix an embedding $F\hookrightarrow\Qb$. By a theorem of Tate it
is known that the group $H^2(G_\Q,\Qb^*)$ is trivial (here $G_\Q$
acts trivially in $\Qb$), so there exist continuous maps
$\alpha\colon G_\Q\to\Qb^*$ such that
$c_B(\sigma,\tau)=\alpha(\sigma)\alpha(\tau)\alpha(\sigma\tau)^{-1}$
for all $\sigma,\tau\in G_\Q$; two such maps differ by a Galois
character. The map $\overline\alpha\colon G_\Q\to\Qb^*/F^*$
obtained viewing the values of $\alpha$ modulo elements of $F^*$
is a morphism; let $K_\alpha$ denote the fixed field of its
kernel, which is an abelian extension of $\Q$. Using the identity
\eqref{csquare-equal-degree} we see that the map
$\varepsilon_\alpha(\sigma)=\alpha(\sigma)^2/\delta(\mu_\sigma)$
is a Galois character $G_\Q\to\Qb^*$; two such characters differ
by the square of a Galois character. Let $K_{\varepsilon_\alpha}$
be the fixed field of $\ker \varepsilon_\alpha$; the fact that
$\delta(\mu_\sigma)$ is real implies that
$K_{\varepsilon_\alpha}\subseteq K_\alpha$. Let
$E_\alpha=F(\{\alpha(\sigma)\}_{\sigma\in G_\Q})$ be the number
field generated over $F$ by the values of $\alpha$; from the
identity defining $\varepsilon_\alpha(\sigma)$ it easily follows
that $E_\alpha/F$ is an abelian extension. Even though the
splitting maps $\alpha$ depend on the cocycle $c_B$ (or, what is
the same, on a system of isogenies between conjugates
of $B$) the morphisms $\overline\alpha$, the fields $K_\alpha$ and
$E_\alpha$, and the characters $\varepsilon_\alpha$ do not depend
on that choice. We will call the maps $\alpha$ \emph{splitting
maps}, the fields $K_\alpha$ \emph{splitting fields}, and the
characters $\varepsilon_\alpha$ \emph{splitting characters} for
the building block $B$. The isogeny class of a building block
determines a set of morphisms
$\overline\alpha\in\Hom(G_\Q,\Qb^*/F^*)$ that is an orbit by the
action of the group of Galois characters $\Hom(G_\Q,\Qb^*)$, and a
set of splitting characters
$\varepsilon_\alpha\in\Hom(G_\Q,\Qb^*)$ that is an orbit by the
action of the subgroup of squares $\Hom(G_\Q,\Qb^*)^2$.

For every Galois character $\varepsilon\colon G_\Q\to\Qb^*$
choose square roots of its values and define
$$c_\varepsilon(\sigma,\tau)=\sqrt{\varepsilon(\sigma)}\sqrt{\varepsilon(\tau)}
    \sqrt{\varepsilon(\sigma\tau)}^{\,\,-1}.$$
This is a $2$-cocycle on $G_\Q$ with values in $\{\pm1\}$. Its
cohomology class $[c_\varepsilon]\in
H^2(G_\Q,\{\pm1\})\simeq\Br_2(\Q)$ gives the obstruction to the
existence of a square root of $\varepsilon$. If two characters
$\varepsilon,\varepsilon'$ differ by the square of a character,
then $[c_\varepsilon]=[c_{\varepsilon'}]$. If $\xi=[c_B]\in
H^2(G_\Q,F^*)$ is the class attached to a building block $B$, then
$\xi_\pm=[c_\varepsilon]$ with $\varepsilon$ any splitting
character for $B$ (see~\cite[Theorem 2.6]{quer-MC}).

\begin{theorem}\label{theorem:existence of strongly modular abelian varieties}
Let $B/\Qb$ be a building block and let $K/\Q$ an abelian extension.
There exists an abelian variety isogenous to $B$
that is completely defined and strongly modular over the field $K$
if and only if $K$ contains a splitting field for $[c_B]$.
\end{theorem}

\begin{proof}
The proof is essentially the same given in \cite[Proposition 5.2]{quer-LMS}
for the case of $\Q$-curves.

Suppose that $K$ contains the splitting field $K_\alpha$
corresponding to some splitting map $\alpha$.
For every element $s\in\Gal(K/\Q)$ choose an element $\alpha(s)$ as any of the values
$\alpha(\sigma)$ for $\sigma\in G_\Q$ an automorphism restricting to $s$,
and define $c(s,t)=\alpha(s)\alpha(t)\alpha(st)^{-1}$.
Then $[c]$ is an element of $H^2(K/\Q,F^*)$ whose inflation equals  $[c_B]$.
By Proposition \ref{proposition-cB-inflation} there exists an abelian variety $B_0$
isogenous to $B$ that is completely defined over the field $K$ and with $[c_{B_0/K}]=[c]$.
By construction the cocycle $c$ is symmetric, hence $[c_{B_0/K}]\in\Ext(K/\Q,F^*)$ and
by Theorem \ref{main-theorem} the variety $B_0$ is strongly modular over $K$.

Conversely, assume that there is a variety isogenous to $B$ that
is strongly modular over the field $K$. Let $c_{B_0/K}$ be a
cocycle on $G=\Gal(K/\Q)$ attached to this variety. Then by
Theorem \ref{main-theorem} the algebra $F^{c_{B_0/K}}[G]$ is
commutative. Hence the $\Qb$-algebra
$\Qb^{c_{B_0/K}}[G]=\Qb\otimes_F F^{c_{B_0/K}}[G]$ is also
commutative, and by a property of twisted group algebras over
algebraically closed fields (cf. \cite[Chapter 2, Corollary
2.5]{karpilovski}) it follows that the image of the class
$[c_{B_0/K}]$ in the Schur multiplier group $H^2(G,\Qb^*)$ is
trivial. Hence there exists a map $s\mapsto\alpha(s)\colon
G\to\Qb^*$ such that
$c_{B_0/K}(s,t)=\alpha(s)\alpha(t)\alpha(st)^{-1}$ and its
inflation to the group $G_\Q$ is a splitting map for the variety
that factors through the group $G$, hence $K_\alpha\subseteq K$.
\end{proof}

\section{QM Jacobian surfaces}

In this section we illustrate the previous general results with applications to the
study of concrete abelian surfaces with quaternionic multiplication.
The surfaces we will be dealing with are obtained as  Jacobians
of a family of genus two curves defined in~\cite{ha-mu}.
We will use some results on their arithmetic that appear in~\cite{ba-gr}
to compute the cocycles needed for the characterization of their strong modularity,
and for the computation of $\Q$-endomorphism algebras of their restriction of scalars.

Since we will need quadratic twists later, we begin with a technical lemma describing the effect
of such a twist in the cohomology classes of interest.
For every abelian variety $B/K$ over a number field $K$
and element $\gamma\in K^*$,
let $B_\gamma$ denote the $K(\sqrt\gamma)$-quadratic twist of the variety $B$ over $K$.
In the standard classification of twists by elements of the first Galois cohomology group with values
in the automorphism group of the object,
this variety corresponds to  the homomorphism in $H^1(G_K,\{\pm1\})$
whose kernel has $K(\sqrt\gamma)$ as  fixed field,
which is given by the formula $\sigma\mapsto\acc\sigma{\sqrt\gamma}/\sqrt\gamma$.
Note that here we interpret $\pm1$ as automorphisms of $B$.
In other words, $B_\gamma$ is the abelian variety determined up to $K$-isomorphism by
the fact that there exists an isomorphism $\phi\colon B_\gamma\to B$ defined over $K(\sqrt\gamma)$
such that $\phi\circ\acc\sigma\phi^{-1}=\acc\sigma{\sqrt\gamma}/\sqrt\gamma$
for every $\sigma\in G_K$.

For hyperelliptic Jacobians the quadratic twists are easily computed:
if $C$ is a hyperelliptic curve defined by the equation $Y^2=F(X)$ then for every
$\gamma\in K^*$ the equation $\gamma Y^2=F(X)$ defines an hyperelliptic curve that
is the $K(\sqrt\gamma)$-quadratic twist of $C$ over $K$.
The Jacobian $\operatorname{Jac(B_\gamma)}$ is the $K(\sqrt\gamma)$-quadratic twist
of the abelian variety $\operatorname{Jac}(B)$ over $K$.

\begin{lemma}\label{lemma-quadratic-twists-and-classes}
Let $B/K$ be a $\Q$-variety completely defined over a Galois number field $K$,
and let $\gamma\in K^*$.
The twist $B_\gamma$ is completely defined over $K$ if and only if
the field $K(\sqrt\gamma)$ is Galois over $\Q$.
In this case, $[c_{B/K}]$ and $[c_{B_\gamma/K}]$ differ
by the cohomology class in $H^2(\Gal(K/\Q),\{\pm1\})$
corresponding to the group extension given by the exact sequence
\begin{equation}\label{eq:exact-seqeunce-of-galois-groups}
\diagram 1\rto&\Gal(K(\sqrt\gamma)/K)\simeq\{\pm1\}\rto&\Gal(K(\sqrt\gamma)/\Q)\rto&\Gal(K/\Q)\rto&1\enddiagram.
\end{equation}
In particular, quadratic twisting  affects only the sign components
of the cohomology classes and leaves the degree components unchanged.
\end{lemma}

\begin{proof}
Note that the cohomology classes $[c_{B/K}]$ and $[c_{B_\gamma/K}]$ we want to compare
take values in groups $F^*$ consisting
of automorphisms of the varieties. The cohomology class attached to the group extension
of the statement takes
values in the group $\{\pm1\}$, which must be identified with a subgroup of $F^*$
by the (canonical) identification of its elements as automorphisms of the variety.

Let $\phi\colon B_\gamma\to B$ be the isomorphism corresponding to
the twist. Then $\phi^{-1}$ is an isomorphism giving $B$ as the
$K(\sqrt\gamma)$-twist of $B_\gamma$ and for every $\sigma\in
G_\Q$ the map
$\acc\sigma\phi\colon\acc\sigma(B_\gamma)\to\acc\sigma B$ is an
isomorphism giving $\acc\sigma(B_\gamma)$ as the
$K(\sqrt{\acc\sigma\gamma})$-twist of $\acc\sigma B$. Every
 isogeny $\nu_\sigma\colon\acc\sigma B_\gamma\to
B_\gamma$ compatible with $\End(B)$ is of the form
$\nu_\sigma=\phi^{-1}\circ\mu_\sigma\circ\acc\sigma\phi$ for an
 isogeny $\mu_\sigma\colon\acc\sigma B\to B$ compatible with $\End(B)$, which by
hypothesis is defined over $K$. For $\tau\in G_K$ one has
$$\acc\tau\nu_\sigma=\acc\tau\phi^{-1}\circ\acc\tau\mu_\sigma\circ\acc{\tau\sigma}\phi
=\acc\tau\phi^{-1}\circ\phi\circ\nu_\sigma\circ\acc\sigma\phi^{-1}\circ\acc{\tau\sigma}\phi=
\frac{\sqrt\gamma}{\acc\tau{\sqrt\gamma}}\circ\nu_\sigma
    \circ\frac{\sqrt{\acc\sigma\gamma}}{\acc\tau{\sqrt{\acc\sigma\gamma}}},$$
which equals $\nu_\sigma$ if and only if the two other maps, each equal to $\pm1$, coincide.
But
$$\frac{\sqrt\gamma}{\acc\tau{\sqrt\gamma}}=
    \frac{\sqrt{\acc\sigma\gamma}}{\acc\tau{\sqrt{\acc\sigma\gamma}}}\quad\Leftrightarrow\quad
    \frac{\sqrt\gamma}{\sqrt{\acc\sigma\gamma}}=
    \frac{\acc\tau{\sqrt\gamma}}{\acc\tau{\sqrt{\acc\sigma\gamma}}}\quad\Leftrightarrow\quad
    \tau\quad\text{fixes}\quad\sqrt{\acc\sigma\gamma}/\sqrt\gamma.$$
Hence the isogeny $\nu_\sigma$ is defined over $K$ if and only if
$\sqrt{\acc\sigma\gamma}/\sqrt\gamma\in K$, and this condition is satisfied for every $\sigma\in G_\Q$
exactly when the extension $K(\sqrt\gamma)/\Q$ is Galois.

Now assume the condition is satisfied.
For each $s\in\Gal(K/\Q)$ fix a  lift $\tilde s$ of $s$ in $\Gal(K(\sqrt\gamma)/\Q)$. Then
\begin{eqnarray*}
c_{B_\gamma/K}(s,t)&=&\nu_{\tilde s}\circ{^{\tilde s}\nu_{\tilde t}}\circ\nu_{\widetilde{st}}^{-1}
=\phi^{-1}\circ\mu_{\tilde s}\circ{^{\tilde s}\phi}\circ{^{\tilde s}\phi}^{-1}\circ{^{\tilde s}
\mu_{\tilde t}}\circ{^{\tilde s\tilde t}\phi}\circ{^{\widetilde{st}}\phi}^{-1}\circ\mu_{\widetilde{st}}^{-1}\circ\phi\\
&=&\phi^{-1}\circ\mu_{\tilde s}\circ{^{\tilde s}\mu_{\tilde t}}\circ{^{\widetilde{st}}({^{\widetilde{st}^{-1}\tilde s\tilde t
}\phi}\circ \phi^{-1}})\circ\mu_{\widetilde{st}}^{-1}\circ\phi\\
&=&\phi^{-1}\circ\mu_{\tilde s}\circ{^{\tilde s}\mu_{\tilde
t}}\circ\mu_{\widetilde{st}}^{-1}\circ\phi\circ({^{\widetilde{st}^{-1}\tilde
s\tilde t }\phi}\circ \phi^{-1})\\
&=& c_{B/K}(s,t)\circ({^{\widetilde{st}^{-1}\tilde s\tilde t
}\phi}\circ\phi^{-1})=c_{B/K}(s,t)\cdot\frac{^{\widetilde{st}^{-1}\tilde s\tilde t }\sqrt{\gamma}}
{\sqrt{\gamma}},
\end{eqnarray*}
and the factor in the right is a cocycle
associated to the group extension~\eqref{eq:exact-seqeunce-of-galois-groups}.
\end{proof}

We now recall some notation and results from~\cite{ba-gr}. If
$a,b\in\Q$ we denote $(a,b)_\Q$ the quaternion algebra over $\Q$
generated by $\imath,\jmath$ with $\imath^2=a$, $\jmath^2=b$ and
$\imath\jmath+\jmath\imath=0$. Let $\B_6=(2,3)_\Q$ be the
quaternion algebra of discriminant 6 over $\Q$, and let
$\O=\Z[\imath,(1+\jmath)/2]$, which is a maximal order of $\B_6$.
We also define the subrings $R_2=\Z[\imath]\simeq\Z[\sqrt2]$,
$R_3=\Z[\jmath+\imath\jmath]\simeq\Z[\sqrt -3]$ and
$R_6=\Z[\mu]\simeq\Z[\sqrt{6}]$, where
$\mu=2\jmath+\imath\jmath$. A curve $C$ is said to be a
\emph{QM-curve with respect to $\O$} if $\O$ can be embedded into
the endomorphism ring of its Jacobian. If $(B,\rho)/\Qb$ is a
polarized abelian variety and $R$ is a subring of $\End(B)$, the
\emph{field of moduli $k_R$} is defined to be the smallest number
field such that for any $\sigma \in \Gal(\Qb/k_R)$ there exists an
isomorphism $\phi_\sigma\colon\acc\sigma B \rightarrow B$ with
$\phi_\sigma^*(\rho)=\acc\sigma\rho$ and such that $r\circ
\phi_\sigma=\phi_\sigma\circ\acc\sigma r$ for all $r\in R$. In
other words, $k_R$ is the field of moduli of the object consisting
of the polarized abelian variety $(B,\rho)$ together with the ring
of endomorphisms $R\subseteq\End(B)$.

The family of surfaces we are going to consider is the following.
For every algebraic number $j\in\Qb$ let $C_j$ be the genus $2$ curve with equation
\begin{equation}\label{eq:family of curves Cj}
\begin{aligned}
C_j\colon\quad Y^2=\ &\left(-4+3\sqrt{-6j}\right)X^6-12(27j+16)X^5-6\,(27j+16)\left(28+9\sqrt{-6j}\right)X^4\\
&+16(27j+16)^2X^3+12(27j+16)^2\left(28-9\sqrt{-6j}\right)X^2\\
&-48(27j+16)^3X+8(27j+16)^3\left(4+3\sqrt{-6j}\right)
\end{aligned}
\end{equation}
Let $B_j=\operatorname{Jac}(C_j)$ be its Jacobian with the canonical principal polarization induced by $C_j$.
Some properties of these objects proved in \cite{ba-gr} are summarized in the following statement.

\begin{theorem}[Baba-Granath]{\ }
\begin{itemize}
\item The curve $C_j$ has field of moduli $\Q(j)$.
For every $\sigma\in G_{\Q(j)}$ such that $\accl\sigma{\sqrt{-6j}}=-\sqrt{-6j}$ the map
$(x,y)\mapsto\left(\frac{-2(27j+16)}x,\frac{y(-2(27j+16))^{3/2}}{x^3}\right)$
is an isomorphism $\accl\sigma C_j\to C_j$.
\item The field $\Q(\sqrt{-6j})
k_\O$ is a field of definition of the endomorphisms of $B_j$.
\item The curves $C_j$ are $QM$-curves with respect to $\O$.
Moreover, for all $j\in \Q$ but for 26 values, $\End^0(B_j)\simeq\B_6$.
Under this isomorphism the Rosati involution $'$ attached to the canonical polarization of $B_j$
is given by $\varphi'=\mu^{-1}\varphi^*\mu$, where $^*$ indicates the
canonical conjugation of $\B_6$.
\item The fields of moduli $k_R$ for the canonically polarized Jacobian $B_j$
and several rings of endomorphisms of interest are given in the following diagram:
$$
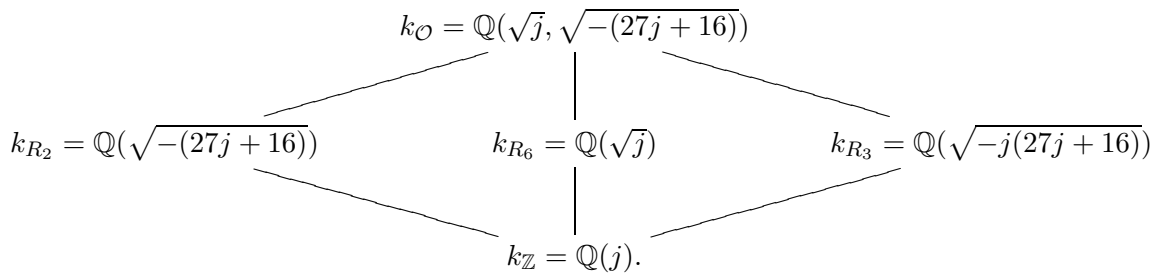
\diagram
&k_\O=\Q(\sqrt{j},\sqrt{-(27j+16)})\dline\dlline\drline&\\
k_{R_2}=\Q(\sqrt{-(27j+16)})&k_{R_6}=\Q(\sqrt{j})&k_{R_3}=\Q(\sqrt{-j(27j+16)})\\
&k_\Z=\Q(j).\ulline\uline\urline&
\enddiagram$$
\end{itemize}
\end{theorem}

When $j\in\Q$ the abelian surfaces $B_j$ have the property that
for every $\sigma\in G_\Q$ there exists an isomorphism
$\phi_\sigma\colon\acc\sigma{B_j}\rightarrow B_j$,
but this isomorphism does not need to be compatible with $\End(B_j)$ (indeed, in general it is not compatible with
$\End(B_j)$).
However, if the algebra of endomorphisms of $B_j$ is isomorphic to $\B_6$ then we
can always find  isogenies compatible with $\End(B_j)$. More generally, we have the following result.

\begin{proposition}\label{proposition-there exist compatible isogenies}
Let $B/\Qb$ be an abelian variety whose algebra of endomorphisms is a central simple $\Q$-algebra.
Let $\sigma\in G_\Q$.
If $\ \acc\sigma{B}$ and $B$ are isogenous
then there exists an isogeny $\acc\sigma{B}\rightarrow B$ compatible with $\End(B)$.
\end{proposition}

\begin{proof}
Call $\D$ the endomorphism algebra of $B$, and let
$\phi_\sigma\colon\acc\sigma{B}\rightarrow B$ be an isogeny. The map
$\varphi\mapsto\phi_\sigma\circ\acc\sigma{\varphi}\circ\phi_\sigma^{-1}$ is a
$\Q$-algebra  automorphism of $\D$ since it fixes the center,
which is $\Q$ by hypothesis. Then the Noether--Skolem Theorem
implies that it is inner; that is, there exists an element
$\psi_\sigma\in\D$ such that
$\phi_\sigma\circ\acc\sigma{\varphi}\circ\phi_\sigma^{-1}=\psi_\sigma^{-1}\circ
\varphi\circ\psi_\sigma$. Then the isogeny
$\mu_\sigma=\psi_\sigma\circ\phi_\sigma$ is compatible with $\D$.
\end{proof}

Hence, we see that if $j$ belongs to $\Q$ and the endomorphism algebra of
$B_j$ is isomorphic to $\B_6$, then $B_j$ is a  building block
completely defined over the field
$$K=\Q(\sqrt{-6j},\sqrt{j},\sqrt{-(27j+16)},\sqrt{-2(27j+16)}).$$
From now on we assume that $\End(B_j)\otimes \Q\simeq \B_6$. Now we aim to compute
the cohomology class $[c_{B_j}]$. The degree component
$\overline{[c_{B_j}]}$ belongs to $\Hom(G_\Q,\Q^*/\{\pm 1
\}\Q^{*2})$, and we use the following notation to indicate the
elements of this group: for $t,d\in\Q^*$ we denote by $(t,d)_P$ the
homomorphism that sends an element $\sigma\in G_\Q$ to $d\cdot\{\pm 1\}\Q^{*2}$
if $\acc\sigma{\sqrt t}=-\sqrt t$ and has trivial image otherwise.
An expression of the form $(t_1,d_1)_P\cdot (t_2,d_2)_P\cdot{\dots}\cdot(t_r,d_r)_P$
denotes the product of such homomorphisms,
and all elements in $\Hom(G_\Q,\Q^*/\{\pm 1\}\Q^{*2})$ admit a
(non-unique) expression of this kind.

\begin{proposition}
The degree and sign components of $[c_{B_j}]$ are given by
\begin{equation}\label{eq:component grau de c_Bj}
\overline{[c_{B_j}]}=(-(27j+16),3)_P\cdot (-j(27j+16),2)_P,
\end{equation}
\begin{equation}\label{eq:component signe de C_Bj}
[c_{B_j}]_\pm=(-(27j+16),3)_\Q\cdot (-j(27j+16),2)_\Q\cdot(2,3)_\Q.
\end{equation}
\end{proposition}

\begin{proof}
Recall that the degree component is the map
$\sigma\mapsto\delta(\mu_\sigma)\ \mathrm{mod}\ \{\pm 1\}\Q^{*2}$,
where $\mu_\sigma$ is any isogeny $\mu_\sigma\colon\acc\sigma{B_j}\rightarrow B_j$ compatible with $\End(B_j)$.
If $\sigma\in\Gal(\Qb/k_\O)$, by the definition of $k_\O$ there
exists an isomorphism
$\phi_\sigma\colon\acc\sigma{B_j}\rightarrow B_j$ compatible with $\End(B_j)$ such that
$\phi_\sigma^*(\rho)=\acc\sigma{\rho}$, where $\rho$ is the
polarization of $B_j$ given by $C_j$. Applying the definition of
$\delta$ (see~\cite[p.~220]{pyle}) we find that
$$\delta(\phi_\sigma)=\phi_\sigma\circ\acc\sigma{\rho^{-1}}\circ\hat\phi_\sigma\circ\rho=
\phi_\sigma\circ\phi_\sigma^{-1}\circ\rho^{-1}\circ\hat\phi_\sigma^{-1}\circ\hat\phi_\sigma\circ\rho=1.$$
Hence, the degree component is the inflation of a map defined in
$\Gal(k_\O/\Q)$. Now, since $k_\O=k_{R_2}\cdot k_{R_3}$, we just
need to compute $\delta(\mu_\sigma)$ for $\sigma\in\Gal(\Qb/k_{R_d})$ for $d=2,3$.

Let $\sigma$ be an element in $\Gal(\Qb/k_{R_d})$ that does not
fix $k_\O$. Since $k_{R_2}=\Q(\sqrt{-(27j+16)})$ and
$k_{R_3}=\Q(\sqrt{-j(27j+16)})$, in order to
prove~\eqref{eq:component grau de c_Bj} we just need to see that
$\delta(\mu_\sigma)\equiv d\,\pmod{\Q^{*2}}$. By the
definition of $k_{R_d}$ there exists an isomorphism
$\phi_\sigma\colon\acc\sigma{B}\rightarrow B$ compatible with the
endomorphisms in $R_d$, but not necessarily compatible with all
the endomorphisms. However, we know from Proposition~\ref{proposition-there
exist compatible isogenies} that we can find $\psi_\sigma\in \B_6$
such that $\mu_\sigma=\psi_\sigma\circ\phi_\sigma$ is an
isogeny compatible with all the endomorphisms. Moreover, from the proof of this  proposition we see that $\psi_\sigma$
is characterized by the
property that
$$\phi_\sigma\circ\acc\sigma{\varphi}\circ\phi_\sigma^{-1}=\psi_\sigma^{-1}\circ
\varphi\circ\psi_\sigma, \quad  \text{for every }\varphi\in \B_6.$$ But if we take
$\varphi\in R_d$, this particularizes to
$\varphi=\psi_\sigma^{-1}\circ \varphi\circ\psi_\sigma$, so
$\psi_d$ commutes with every element in $R_d$, which implies that
$\psi_d$ belongs to $R_d\otimes\Q$. Hence, if we write
$R_d=\Z[c_d]$, with $c_2=\imath$ and $c_3=\jmath+\imath\jmath$ we
have that $\psi_d=a+bc_d$ for some $a,b\in \Q$. In fact, $b\neq 0$
because otherwise the isomorphism $\psi_\sigma$ would be
compatible with all the endomorphisms of $B_j$, and this is not
the case since we are assuming that $\sigma $ does not fix $k_\O$.
Using the definition of $\delta(\mu_\sigma)$ we see that
\begin{eqnarray*}
\delta(\mu_\sigma)&=&\delta(\psi_\sigma\circ\phi_\sigma)=\psi_\sigma\circ\phi_\sigma\circ
{^\sigma\rho^{-1}}\circ\widehat{\psi_\sigma\circ\phi_\sigma}\circ\rho=
\psi_\sigma\circ\phi_\sigma\circ\phi_\sigma^{-1}\circ\rho^{-1}\circ\hat\phi_\sigma^{-1}\circ\hat\phi_\sigma\circ\hat\psi_\sigma\circ\rho\\&=&
\psi_\sigma\circ\rho^{-1}\circ\hat\psi_\sigma\circ\rho=\psi_\sigma\circ\psi_\sigma'.
\end{eqnarray*}
Now we know that if $\varphi\in\B_6$ its Rosati involution is given
by $\varphi'=\mu^{-1}\varphi^*\mu$. Hence,
\begin{eqnarray*}
\delta(\mu_\sigma)&=&\psi_\sigma\circ
\psi_\sigma'=(a+bc_d)(a+bc_d)'=(a+bc_d)\mu^{-1}(a-bc_d)\mu\\
&=&(a+bc_d)^2=a+db^2+2abc_d,
\end{eqnarray*}
and since $\delta(\mu_\sigma)$ must lie in $\Q^*$ and $b\neq0$,
we see that $a=0$ and $\delta(\mu_\sigma)\equiv d\pmod{\Q^{*2}}$.

Now, to prove the identity~\eqref{eq:component signe de C_Bj} we
use~\cite[Theorem 2.8]{quer-MC}, which gives a formula for the
Brauer class of the endomorphism algebra of a building block.
Specialized to our case, and having computed the degree
component, this formula gives
\begin{equation*}
(2,3)_\Q=[c_{B_j}]_\pm\cdot(-(27j+16),3)_\Q\cdot(-j(27j+16),2)_\Q,
\end{equation*}
and from this~\eqref{eq:component signe de C_Bj} follows.
\end{proof}

\subsection*{ A concrete example: $j=1/81$.}
Let us now consider the example corresponding to this value of the parameter;
let $C=C_j$ and $B=\operatorname{Jac}(C)$. We remark that $B$ is QM and not CM, i.e., $\End(B)\otimes\Q\simeq\B_6$.
Then $B$ is a building block completely defined over $K=\Q(\sqrt{-6},\sqrt{-3})$
and it is strongly modular over $K$ if and only if $[c_{B/K}]\in\Ext(K/\Q,\Q^*)$;
that is, if and only if it can be represented by a symmetric cocycle.
In fact, since the degree component is always symmetric (over an abelian extension)
we need to check this property only for the sign component $[c_{B/K}]_\pm\in H^2(K/\Q,\{\pm1\})$.

For $d=-3,-6$ we denote by $\varepsilon_d$ the non-trivial character
$\varepsilon_d\colon\Gal(\Q(\sqrt{d})/\Q)\rightarrow \{\pm1\}$.
The group $H^2(K/\Q,\{\pm 1\})$ admits a basis as a $\Z/2\Z$-vector space consisting of
the classes of three 2-cocycles that we call
$c_{\varepsilon_{-6}},\ c_{\varepsilon_{-3}}$ and $c_{-6,-3}$
(see for instance~\cite[Section 2]{quer-JA} for the definition of these
cocycles and their properties). Hence, we have that
\begin{equation}\label{eq:expresion-for-cB/K-in-terms-of-the-basis}
[c_{B/K}]_\pm= [c_{\varepsilon_{-6}}]^a \cdot
[c_{\varepsilon_{-3}}]^b \cdot [c_{-6,-3}]^c
\end{equation}
for some $a,b,c\in\{0,1\}$, and $[c_{B/K}]$ lies in $ \Ext(K/\Q,\{\pm
1\})$ if and only if $c=0$. We know that $\Inf
[c_{B/K}]_\pm=[c_B]_\pm$, which in this case turns out to be trivial
by~\eqref{eq:component signe de C_Bj}. Since $\Inf
[c_{\varepsilon_{-6}}]=(-6,-1)_\Q$, $\Inf
[c_{\varepsilon_{-3}}]=(-3,-1)_\Q$ and $\Inf
[c_{-6,-3}]=(-6,-3)_\Q$, the only possibilities are $a=b=c=0$ or
$a=b=1,\ c=0$. In both cases $c=0$, which implies  that
$[c_{B/K}]_\pm$ lies in $\Ext(K/\Q,\{\pm 1\})$ and therefore $B$ is
strongly modular over $K$.

Let $\xi_1,\xi_2\in H^2(K/\Q,\Q^*)$ be the cohomology classes
with degree component $\overline\xi_1=\overline\xi_2=(-3,6)_P$
and sign component $\xi_{1\pm}=1$, $\xi_{2\pm}=[c_{\varepsilon_{-6}}]\cdot[c_{\varepsilon_{-3}}]$,
and let $A=\Res_{K/\Q}B$.
We have seen that either $[c_{B/K}]=\xi_1$ or $[c_{B/K}]=\xi_2$.
By direct computation we see that $$\Q^{\xi_1}[G]\simeq \Q(\sqrt{6})\times\Q(\sqrt{6})\quad \text{
and}\quad \Q^{\xi_2}[G]\simeq\Q(\sqrt{6},\sqrt{-6}),$$ where $G=\Gal(K/\Q)$.
By Proposition~\ref{proposition-endomorphism-algebra-restriction-of-scalars}
we see that  if $[c_{B/K}]=\xi_1$ then $A\sim_\Q A_g^2\times A_h^2$
for some newforms $g$ and $h$ with $\End_\Q(A_g)\simeq\End_\Q(A_h)\simeq\Q(\sqrt{6})$.
On the other hand, if $[c_{B/K}]=\xi_2$ then $A\sim_\Q A_f^2$ for some newform $f$ with
$\End_\Q(A_f)\simeq\Q(\sqrt{6},\sqrt{-6})$.
If $\p$ is a prime of $K$, let $L_\p(B/K,T)$ be the numerator of the zeta function of
the reduction of $B$ modulo $\p$, and let
$$L_p(B/K,T)=\prod_{\p\,|\, p}L_\p(B/K,T^{\mathrm{N}\p}),$$ which is
in fact equal to $L_p(A/\Q,T)$. By counting points in the
reduction of $B$ modulo primes of $K$ we have computed some of
these local factors:
\[
\begin{array}{|l|l|}
\hline p & L_p(B/K,T)^{-1}=L_p(A/\Q,T)^{-1}\\ \hline
5&(1-4T^2+5^2T^4)^4 \\
\hline
7&(1-2T+7T^2)^8 \\
\hline
11&(1-16T^2+11^2T^4)^4 \\
\hline
13&(1-25T^2+13^2T^4)^4 \\
\hline
17&(1-20T^2+17^2T^4)^4 \\
\hline
19&(1-37T^2+19^2T^4)^4 \\
\hline
23&(1+40T^2+23^2T^4)^4 \\
\hline
29&(1-34T^2+29^2T^4)^4 \\
\hline
31&(1-T+31T^2)^8  \\
\hline
37&(1-10T^2+37^2T^4)^4 \\
\hline
41&(1+58T^2+41^2T^4)^4 \\
\hline
\end{array}\]
Some of these factors are of the form $(1+e_pT^2+p^2T^4)^4$,
and if we had $A\sim_\Q A_g^2\times A_h^2$ for some newforms
$g=\sum b_nq^n$ and $h=\sum c_nq^n$ this would imply that
$$1+e_pT^2+p^2T^4=(1-b_pT+pT^2)(1-\acc\sigma b_pT+pT^2),$$
being $\sigma$ the non-trivial automorphism of $\Q(\sqrt 6)/\Q$.
 A similar relation would hold for the coefficients $c_p$. But
this relation implies that $b_p^2=c_p^2=2p-e_p$, which is
impossible for the computed values of $e_p$ because then the
coefficients $b_p$ and $c_p$ would not lie in $\Q(\sqrt 6)$.
Therefore, the actual cohomology class is $[c_{B/K}]=\xi_2$ and
$A\sim_\Q A_f^2$ for some newform $f=\sum a_nq^n$ with the $a_n$
generating $ \Q(\sqrt 6,\sqrt{-6})$. However,
Proposition~\ref{proposition-cB-restriction} tells us that there
also exists a variety in the $\Qb$-isogeny class of $B$ completely
defined over $K$ and with cohomology class $\xi_1$. We will find
such a variety as the Jacobian of a  quadratic twist of $C$.

Let $\gamma=2-\sqrt 2$. The extension $K(\sqrt \gamma)/\Q$ is Galois,
and an easy computation shows that the cohomology class associated
to~\eqref{eq:exact-seqeunce-of-galois-groups} in this particular case is
$[c_{\varepsilon_{-6}}]\cdot [c_{\varepsilon_{-3}}]$.
Hence, the variety $B_\gamma$ is completely defined over $K$ and
$[c_{B_\gamma/K}]=c_{B/K}\cdot [c_{\varepsilon_{-6}}]\cdot
[c_{\varepsilon_{-3}}]=\xi_1$. Arguing as before we see that
$A_\gamma=\Res_{K/\Q}B_\gamma$ is $\Q$-isogenous to the square of
a product of two modular abelian varieties with field of Fourier
coefficients equal to $\Q(\sqrt 6)$. In $S_2(\Gamma_0(2^4\cdot
3^5))$ we find a newform with field of Fourier coefficients
$\Q(\sqrt 6)$ and Fourier expansion
$$g=q+\sqrt6 q^5-2q^7+\sqrt6 q^{11}-q^{13}+3\sqrt6q^{17}+q^{19}-\sqrt6 q^{23}+\dots$$
Let $\varepsilon$ be the quadratic Dirichlet character of conductor $8$
satisfying $\varepsilon(3)=\varepsilon(5)=-1$. Let $h=g\otimes\varepsilon$, which is a newform in
$S_2(\Gamma_0(2^6\cdot3^5))$.
In the following table we list some local factors of the $L$-functions
corresponding to the varieties $B_\gamma/K$, $A_g/\Q$ and $A_h/\Q$.
$$
\begin{array}{|l|l|l|l|}
\hline p & L_p(B_\gamma/K,T)^{-1}=L_p(A_\gamma/\Q,T)^{-1} & L_p(A_g/\Q,T)^{-1} & L_p(A_h/\Q,T)^{-1}\\
 \hline
5&(1+4T^2+5^2T^4)^4 & (1+4T^2+5^2T^4) &(1+4T^2+5^2T^4)\\
\hline
7&(1+2T+7T^2)^8 &(1+2T+7T^2)^2 &(1+2T+7T^2)^2\\
\hline
11&(1+16T^2+11^2T^4)^4&(1+16T^2+11^2T^4) &(1+16T^2+11^2T^4)\\
\hline
13&(1-T+13T^2)4(1+T+13T^2)4& (1+T+13T^2)^2 &(1-T+13T^2)^2\\
\hline
17&(1-20T^2+17^2T^4)4&(1-20T^2+17^2T^4) &(1-20T^2+17^2T^4)\\
\hline
19&(1-T+19T^2)^4(1+T+19T^2)^4&(1-T+19T^2)^2 &(1+T+19T^2)^2\\
\hline
23&(1+40T^2+23^2T^4)^4 & (1+40T^2+23^2T^4) &(1+40T^2+23^2T^4)\\
\hline
29&(1+34T^2+29^2T^4)^4& (1+34T^2+29^2T^4) &(1+34T^2+29^2T^4)\\
\hline
31&(1-T+31T^2)^8 & (1-T+31T^2)^2 &(1-T+31T^2)^2\\
\hline
37&(1-8T+37T^2)^4(1+8T+37T^2)^4&(1-8T+37T^2)^2 &(1+8T+37T^2)^2\\
\hline
41&(1+58T^2+41^2T^4)^4&(1+58T^2+41^2T^4) &(1+58T^2+41^2T^4)\\
\hline
\end{array}$$
We have checked the equality of the local factors of the $L$-functions of  $A_\gamma$ and $ A_g^2\times A_h^2$ for all
primes
$p<1000$ ($p\neq 2,3$) and this suggests that $A_\gamma\sim_\Q
A_g^2\times A_h^2$.

Comparing the local factors $L_p(B/K,T)$ and $L_p(B_\gamma/K,T)$
we can also find a modular form $f$ such that $A\sim_\Q A_f^2$ as
a twist of $g$. More precisely, let $\psi$ be the Dirichlet
character of order $4$ and conductor $16$ such that
$\psi(3)=-\sqrt{-1}$ and $\psi(5)=\sqrt{-1}$. The modular form
$f=g\otimes\psi$ is a newform in
$S_2(\Gamma_0(2^8\cdot3^5),\psi^2)$ and the local factors
$L_p(B/K,T)$ and $L_p(A_f,T)^2$ coincide for all primes $p<1000$
($p\neq2,3$).

\subsection*{ A concrete example: $j=-4/27$.}
We now consider another example, corresponding to the stated value of $j$.
The Jacobian $B$ of the curve $C_j$ is also a building
block completely defined over $K=\Q(\sqrt{-6},\sqrt{-3})$,
but a similar analysis shows that in this case the only possibilities
for $[c_{B/K}]_\pm$ have $c=1$ in the
expression~\eqref{eq:expresion-for-cB/K-in-terms-of-the-basis},
and therefore $[c_{B/K}]_\pm$ is not symmetric.
This means that no variety in the isogeny class of $B$ is strongly modular over $K$.
If we let for instance $L=K(\sqrt{-1})$ it is easy to see that
there exist symmetric elements of $H^2(L/\Q,\{\pm 1\})$ whose
inflation to $G_\Q$ is $[c_B]_\pm$, and then
by Proposition~\ref{proposition-cB-inflation} and Theorem~\ref{main-theorem}
there exists a variety isogenous to $B$ completely defined and strongly modular over $L$.

Thus, in this case we have seen that it is enough to go to a quadratic extension $L$ of $K$
to find a variety in the isogeny class of $B$ that is strongly modular over L.
However, in the family $\{B_j\}_{j\in\Q}$ we can
find varieties where any minimal field $L$ with this property is
arbitrarily large. In fact, by Theorem~\ref{theorem:existence of strongly
modular abelian varieties} this is equivalent to find in this
family varieties where the degree of any splitting field is
arbitrarily large. We will see this by means of the following lemma.

\begin{lemma}\label{lemma-orders-of-splitting-characters}
Let $r\geq 2$ be an integer and let $p$ be a prime such that
$p\equiv1\pmod{2^r}$ and $p\equiv-1\pmod3$.
Then the order of any splitting character for $B_{1/p}$ is at least $2^r$.
\end{lemma}

\begin{proof}
For simplicity we call $B$ the variety $B_{1/p}$, and let $[c_B]$ be its attached cohomology class.
By~\eqref{eq:component signe de C_Bj}
the sign component $[c_B]_\pm$ is given as the following product of quaternion algebras:
\begin{equation*}
[c_B]_\pm=(-(27+16p)/p,3)_\Q\cdot (-(27+16p),2)_\Q\cdot (2,3)_\Q.
\end{equation*}
Applying the formulas for computing the local Hilbert Symbols at $p$ we find that
$$(-(27+16p)/p,3)_p=-1,\qquad (-(27+16p),2)_p=1,\qquad (2,3)_p=1,$$
and this implies that the local component of $[c_B]_\pm$ at the prime $p$ is $-1$.
But $[c_B]_\pm=[c_\varepsilon]$, where $\varepsilon$ is the splitting
character  associated to any splitting map $\alpha$ for $B$. We
can identify $\varepsilon$ with a primitive Dirichlet character of
a certain conductor $N$, and if $\varepsilon_p$ denotes the
component of $\varepsilon$ modulo the largest power of $p$
dividing $N$, then the local component of $[c_\varepsilon]$ at $p$
is given by $\varepsilon_p(-1)$.
The value $\varepsilon_p(-1)=-1$ is taken by the characters of order multiple of $2^{\ord_2(p-1)}$,
and it follows that $\ord(\varepsilon)\geq\ord(\varepsilon_p)\geq2^{\ord_2(p-1)}$,
which is at least $2^r$ by our choice of $p$.
\end{proof}

\begin{proposition}
For any integer $r$ there exists a variety $B$ in the family
$\{B_j\}_{j\in\Q}$ such that any splitting field for $B$ has
degree at least $2^r$.
\end{proposition}

\begin{proof}
Take a prime $p$ as in the previous lemma, and take as $B$ the
variety $B_{1/p}$. Let $\alpha$ be any splitting map for $B$, and
let $\varepsilon$ be its associated splitting character. Then we
have that $[K_\alpha:\Q]\geq [K_\varepsilon:\Q]\geq 2^r$.
\end{proof}

From  Lemma~\ref{lemma-orders-of-splitting-characters} we can
derive another interesting consequence.

\begin{proposition}
Let $g$ be any natural number.
There exist varieties $B$ in the family $\{B_j\}_{j\in \Q}$ such that
every $\Q$-simple abelian variety $A$ of $\GL_2$-type having $B$ as its simple factor
is of dimension $\dim A\geq g$.
\end{proposition}

\begin{proof}
Let $r$ be an integer such that $\varphi(2^r)=2^{r-1}\geq g$,
and take $B=B_{1/p}$ with $p$ a prime as in Lemma~\ref{lemma-orders-of-splitting-characters}.
If $A$ is a simple abelian variety of $\GL_2$-type that has $B$ as its simple factor, the field
$E=\End_\Q^0(A)$ is isomorphic to $E_\alpha$ for some splitting map $\alpha$ for $B$.
The field $E_\alpha$ contains the values
of the splitting character $\varepsilon$ associated to $\alpha$. Therefore, it  contains the $2^r$-th cyclotomic
extension
and we have that $\dim A=[E_\alpha:\Q]\geq\varphi(2^r)\geq g$.
\end{proof}



\end{document}